\documentclass{article}%
\usepackage{graphicx}
\usepackage{amsmath}
\usepackage{amsfonts}
\usepackage{amssymb}
\usepackage{setspace}
\usepackage[hmargin=3.5cm,vmargin=3.5cm]{geometry}%
\providecommand{\U}[1]{\protect\rule{.1in}{.1in}}
\providecommand{\U}[1]{\protect\rule{.1in}{.1in}}
\providecommand{\U}[1]{\protect\rule{.1in}{.1in}}
\providecommand{\U}[1]{\protect\rule{.1in}{.1in}}
\providecommand{\U}[1]{\protect\rule{.1in}{.1in}}
\providecommand{\U}[1]{\protect\rule{.1in}{.1in}}
\providecommand{\U}[1]{\protect\rule{.1in}{.1in}}
\providecommand{\U}[1]{\protect\rule{.1in}{.1in}}
\providecommand{\U}[1]{\protect\rule{.1in}{.1in}}
\providecommand{\U}[1]{\protect\rule{.1in}{.1in}}
\providecommand{\U}[1]{\protect\rule{.1in}{.1in}}
\providecommand{\U}[1]{\protect\rule{.1in}{.1in}}
\providecommand{\U}[1]{\protect\rule{.1in}{.1in}}
\providecommand{\U}[1]{\protect\rule{.1in}{.1in}}
\providecommand{\U}[1]{\protect\rule{.1in}{.1in}}
\providecommand{\U}[1]{\protect\rule{.1in}{.1in}}
\providecommand{\U}[1]{\protect\rule{.1in}{.1in}}
\providecommand{\U}[1]{\protect\rule{.1in}{.1in}}
\providecommand{\U}[1]{\protect\rule{.1in}{.1in}}
\providecommand{\U}[1]{\protect\rule{.1in}{.1in}}
\providecommand{\U}[1]{\protect\rule{.1in}{.1in}}
\providecommand{\U}[1]{\protect\rule{.1in}{.1in}}
\providecommand{\U}[1]{\protect\rule{.1in}{.1in}}
\providecommand{\U}[1]{\protect\rule{.1in}{.1in}}
\providecommand{\U}[1]{\protect\rule{.1in}{.1in}}
\providecommand{\U}[1]{\protect\rule{.1in}{.1in}}
\providecommand{\U}[1]{\protect\rule{.1in}{.1in}}
\providecommand{\U}[1]{\protect\rule{.1in}{.1in}}
\providecommand{\U}[1]{\protect\rule{.1in}{.1in}}
\providecommand{\U}[1]{\protect\rule{.1in}{.1in}}
\providecommand{\U}[1]{\protect\rule{.1in}{.1in}}
\providecommand{\U}[1]{\protect\rule{.1in}{.1in}}
\providecommand{\U}[1]{\protect\rule{.1in}{.1in}}
\providecommand{\U}[1]{\protect\rule{.1in}{.1in}}
\providecommand{\U}[1]{\protect\rule{.1in}{.1in}}
\providecommand{\U}[1]{\protect\rule{.1in}{.1in}}
\providecommand{\U}[1]{\protect\rule{.1in}{.1in}}
\providecommand{\U}[1]{\protect\rule{.1in}{.1in}}

\newtheorem{theorem}{Theorem}
{}

\newtheorem{conclusion}{Conclusion}

\newtheorem{conjecture}{Conjecture}
\newtheorem{corollary}{Corollary}

\newtheorem{definition}{Definition}

\newtheorem{notation}{Notation}

\newtheorem{proposition}{Proposition}
\newtheorem{remark}{Remark}

\newtheorem{summary}{Summary}
\newenvironment{proof}[1][Proof]{\textbf{#1.} }{\ \rule{0.5em}{0.5em}}

\begin{document}

\title{Spectral Analysis of the Schrodinger Operator with an Optical Potential}
\author{O. A. Veliev\\{\small Dogus University, Ac\i badem, Kadik\"{o}y, \ }\\{\small Istanbul, Turkey,}\ {\small e-mail: oveliev@dogus.edu.tr}}
\date{}
\maketitle

\begin{abstract}
In this paper we g\i ve a complete description of the spectral analysis of the
Schrodinger operator $L(q)$ with the potential $q(x)=4\cos^{2}x+4iV\sin2x$ for
all $V>1/2.$ First we consider the Bolch eigenvalues and spectrum of $L(q).$
\ Then using it we investigate spectral singularities and essential spectral
singularities (ESS). We prove that the operator $L(q)$ has no ESS and has ESS
respectively if and only if $V\neq V_{k}$ and $V=V_{k}$ for $k\geq2,$ where
$V_{k}\rightarrow\infty$ as $k\rightarrow\infty$, \ $V_{2}$ is the second
critical point and $V_{2}<V_{3}<\cdot\cdot\cdot$. Using it we classify the
spectral expansion in term of the critical points $V_{k}$ for $k\geq2$.
Finally we discuss the critical points, formulate some conjectures and
describe the changes of the spectrum of $L(q)$ when $V$ changes fron $1/2$ to
$\infty.$



\end{abstract}

\section{Introduction}

In this paper we consider in detail the spectrum, spectral singularities, ESS
and spectral expansion of the Schr\"{o}dinger operator
\begin{equation}
L(q)=-\frac{d^{2}}{dx^{2}}+q
\end{equation}
with the optical potential%
\begin{equation}
q(x)=(1+2V)e^{i2x}+(1-2V)e^{-i2x}%
\end{equation}
for all $V>1/2.$ The operator $L(q)$ is defined in $W_{2}^{2}(-\infty,\infty)$
by $L(q)y=-y^{^{\prime\prime}}+qy$ and is a densely defined closed operator in
$L_{2}(-\infty,\infty).$ It is well-known that [1, 8, 11] the spectrum
$\sigma(L(q))$ of $L(q)$ is the union of the spectra $\sigma(L_{t}(q))$ of the
operators $L_{t}(q)$ for $t\in(-1,1]$ generated in $L_{2}[0,\pi]$ by the
expression $-y^{^{\prime\prime}}+qy$ and the boundary conditions
\begin{equation}
y(\pi)=e^{i\pi t}y(0),\text{ }y^{^{\prime}}(\pi)=e^{i\pi t}y^{^{\prime}}(0).
\end{equation}
The spectrum of $L_{t}(q)$ consists of the eigenvalues $\lambda_{1}(t),$
$\lambda_{2}(t),...,$ called as the Bloch eigenvalues of $L(q),$ that are the
roots of the characteristic equation
\begin{equation}
F(\lambda)=2\cos\pi t,
\end{equation}
where $F(\lambda):=\varphi^{\prime}(\pi,\lambda)+\theta(\pi,\lambda)$ is the
Hill discriminant, $\theta$ and $\varphi$ are the solutions of
\begin{equation}
-y^{^{\prime\prime}}(x)+q(x)y(x)=\lambda y(x)
\end{equation}
satisfying the initial conditions $\theta(0,\lambda)=\varphi^{\prime
}(0,\lambda)=1,$ $\theta^{\prime}(0,\lambda)=\varphi(0,\lambda)=0.$ Moreover,
in [19] we proved that the eigenvalues $\lambda_{n}(t)$ of $L_{t}(q)$ can be
numbered (counting the multiplicity) by elements of $\mathbb{Z}$ and hence by
elements of $\mathbb{N}$ such that, for each $n$ the function $\lambda_{n}$ is
continuous on $(-1,1]$ and $\lambda_{n}(t)$ uniformly with respect to
$t\in(-1,1]$ approaches $\infty$ as $n\rightarrow\infty.$ Thus the spectrum of
$L(q)$ is the union of the continuous curves
\begin{equation}
\Gamma_{n}=\{\lambda_{n}(t):t\in(-1,1]\}
\end{equation}
for $n\in\mathbb{N}$ called as the $n$-th bands of the spectrum. By (3) and
(4), $\lambda_{n}(-t)=\lambda_{n}(t)$ and the end points of $\Gamma_{n}$ are
$\lambda_{n}(0)$ and $\lambda_{n}(1)$.

Some physically interesting results have been obtained by considering the
potential (2). \ For the first time, the mathematical explanation of the
nonreality of $\sigma(L(q))$ for $V>0.5$ and finding the streshold $0.5$
(first critical point $V_{1}$) was done by Makris et al [6, 7]. Moreover, for
$V=0,85$ they sketch the real and imaginary parts of the first two bands by
using the numerical methods. In [9] Midya et al reduce the operator $L(q)$
with potential (2) to the Mathieu operator and using the tabular values
establish that there is second critical point $V_{2}\sim0.888437$ after which
no parts of the first and second bands remain real.

In [15] we proved that if $ab=\widetilde{a}\widetilde{b}$, where
$a,b,\widetilde{a},$ and $\widetilde{b}$ are arbitrary complex numbers, then
the operators $L(q)$ and $L(\widetilde{q})$ with potentials
\begin{equation}
q(x)=ae^{-i2x}+be^{i2x}%
\end{equation}
and $\widetilde{q}(x)=\widetilde{a}e^{-i2x}+\widetilde{b}e^{i2x}$ have the
same Hill discriminants $F(\lambda)$ and $G(\lambda),$ Bloch eigenvalues and
spectra. Thus for $q(x)=(1+2V)e^{i2x}+(1-2V)e^{-i2x}$ and $p(x)=2ic\cos2x$ we
have
\begin{equation}
F(\lambda)=G(\lambda),\text{ }\sigma(L(V))=\sigma(H(c)),\text{ }\sigma
(L_{t}(V))=\sigma(H_{t}(c))
\end{equation}
for $t\in\lbrack0,1]$ and $c=\sqrt{4V^{2}-1},$ where for brevity of the
notations, the operators $L_{t}(q),$ $L(q),$ $L_{t}(p)$ and $L(p)$ are denoted
by $L_{t}(V),$ $L(V)$, $H_{t}(c)$ and $H(c)$ respectively. If $V<1/2$ then
$ic$ is a real number and $H(c)$ is the well-known self-adjoint Mathieu
operator. The case $V=1/2$ is also well-known (see [2, 5, 18] and references
therein). Thus we need to consider the operator $L(V)$ in the case $V>1/2$
which is the consideration of the operator $H(c)$ with pure imaginary
potential $2ic\cos2x$ for $c>0$ due to (8).

In the paper [22] we have gave a complete description, provided with a
mathematical proof, of the shape of $\sigma(L(V))$, when $V$ changes from
$1/2$ to $\sqrt{5}/2,$ that is, $c\in(0,2).$ We proved that the second
critical point $V_{2}$ is a number between $0.8884370025$ and $0.8884370117$.
Moreover, it was proven that $V_{2}$ is the unique degeneration point for the
first periodic eigenvalue, in the sense that the first periodic eigenvalue of
the potential (2) is simple for all $V\in(1/2,\sqrt{5}/2)\backslash\left\{
V_{2}\right\}  $ and is double for $V=V_{2}.$ Besides, we gave an approach for
finding the arbitrary close values of the second critical point $V_{2}.$

Let us briefly and visually explain those results of [22] about the spectrum
of $L(V),$ the eigenvalues of $L_{0}(V)$ (periodic eigenvalues $\lambda
_{1}(0)$, $\lambda_{2}(0)$, ...) and the eigenvalues of $L_{1}(V)$
(antiperiodic eigenvalues $\lambda_{1}(1)$, $\lambda_{2}(1),...$).
Antiperiodic eigenvalues are nonreal and simple for all $V\in(1/2,\sqrt{5}%
/2)$. Moreover, they are pairwise conjugate numbers. To describe the periodic
eigenvalues and spectrum of $L(V)$ we consider the following cases. \ 

\textbf{Case 1:} $1/2<V<V_{2},$ where $V_{2}$ is the second critical point.
Then all periodic eigenvalues are real and simple. The bands $\Gamma_{2k-1}$
and $\Gamma_{2k}$ are joined and they form together the connected subset
$\Omega_{k}=:\Gamma_{2k-1}\cup\Gamma_{2k}$ of the spectrum. The spectrum
$\sigma(L(V))$ consists of the pairwise disjoint sets $\Omega_{1},$
$\Omega_{2},...$called as the components of the spectrum. Moreover, the
$(2k-1)$-th and $2k$-th bands $\Gamma_{2k-1}$ and $\Gamma_{2k}$ become to be
connected by their interior point $\lambda_{2k-1}(t_{k})=\lambda_{2k}(t_{k}),$
where $t_{k}\in(0,\pi)$, which is the unique multiple and double Bloch
eigenvalue lying in $\Omega_{k}.$ The real part of $\Gamma_{2k-1}\cup
\Gamma_{2k}$ is the closed interval $I_{k}=:\left[  \lambda_{2k-1}%
(0),\lambda_{2k}(0)\right]  .$ The subintervals $\left[  \lambda
_{2k-1}(0),\lambda_{2k-1}(t_{k})\right]  $ and $\left[  \lambda_{2k-1}%
(t_{k}),\lambda_{2k}(0)\right]  $ of $I_{k}$ are the real parts of
$\Gamma_{2k-1}$ and $\Gamma_{2k}$ respectively. The other parts $\left\{
\lambda_{2k-1}(t):t\in(t_{k},1]\right\}  $ and $\left\{  \lambda_{2k}%
(t):t\in(t_{k},1]\right\}  $ of the bands $\Gamma_{2k-1}$ and $\Gamma_{2k}$
are \ nonreal and symmetric with respect to the real line.

\textbf{Case 2:} $V=V_{2}.$ If $V$ approaches $V_{2}$ from the left then the
first and second periodic eigenvalues get close to each other and the length
of the interval $I_{1}$ approaches zero. As a result we get the equality
$\lambda_{1}(0)=\lambda_{2}(0)$ which means that the first and second bands
$\Gamma_{1}$ and $\Gamma_{2}$ have only one real point which is their common
point. Moreover, it shows that if $V=V_{2}$, then the first and second
periodic eigenvalues coincide and the first periodic eigenvalue becomes double
eigenvalue. Hence the second critical point $V_{2}$ is a degeneration point
for the first periodic eigenvalue. Thus, in the case $V=V_{2}$ the real part
of $\Gamma_{1}$ as well as $\Gamma_{2}$ is a point $\lambda_{1}(0).$

\textbf{Case 3:} $V_{2}<V<\sqrt{5}/2$ . If $V$ moves to the right of $V_{2},$
then the eigenvalues $\lambda_{1}(0)$ and $\lambda_{2}(0)$ get off the real
line and hence $I_{1}$ becomes the empty set, while all other periodic
eigenvalues continue to be real and simple. As a result, the first and second
bands $\Gamma_{1}$ and $\Gamma_{2}$ become nonreal separated curves symmetric
with respect to the real line. The set $\Omega_{1}$ becomes as the union of
the two separated symmetric curves $\Gamma_{1}$ and $\Gamma_{2}$ , while
$\Omega_{2},$ $\Omega_{3},...,$ continue to be connected sets. Summarizing the
Cases 1-3 let us emphasize the following results of [22] that will be used in
the next sections.

\begin{summary}
If $1/2<V<\sqrt{5}/2,$ then the antiperiodic eigenvalues $\lambda_{n}(1)$ for
all $n\in\mathbb{N}$ are the simple eigenvalues and nonreal numbers. If
$1/2<V<V_{2},$ then the periodic eigenvalues $\lambda_{n}(0)$ for all $n$ are
the simple eigenvalues and the real numbers and $\operatorname{Re}\left(
\Gamma_{2k-1}\cup\Gamma_{2k}\right)  =\left[  \lambda_{2k-1}(0),\lambda
_{2k}(0)\right]  $. If$\ V=V_{2},$ then $\lambda_{n}(0)$ for all
$n\in\mathbb{N}$ are real numbers, $\lambda_{1}(0)=\lambda_{2}(0)$ is a double
eigenvalue and $\lambda_{n}(0)$ for $n>2$ are the simple eigenvalues. If
$V_{2}<V<\sqrt{5}/2$, then the eigenvalues $\lambda_{1}(t)$ and $\lambda
_{2}(t)$ are nonreal pairwise conjugate numbers for all $t\in\left[
0,1\right]  ,$ all periodic eigenvalues are simple and $\lambda_{n}%
(0)\in\mathbb{R}$ for $n>2.$
\end{summary}

This paper can be considered as continuation of the paper [22]. Here we
investigate the Bloch eigenvalues and the spectrum of $L(V)$ for all $V>1/2,$
while they was investigated in [22] only for $1/2<V<\sqrt{5}/2.$ Note that as
$V$ increases the influence of the perturbation (2) also increases and the
investigation of the Bloch eigenvalues and spectrum become more complicated.
Moreover, in this paper, we consider the ESS and spectral expansion of $L(V),$
while in [22] they was not considered at all. Thus this paper with [22] give a
complete description of the spectral analysis of $L(V)$ for all $V>1/2.$

First, in Section 2 we define the spectral singularities and ESS and give some
results of the papers [7, 11, 12, 17, 19-24] that are used in the next
sections. Using it in Section 3 we consider the Bloch eigenvalues and spectrum
of $L(V)$. Then in Section 4 using the results of Section 3, we study the
spectral singularity, ESS and investigate the spectral expansion of $L(V)$.
Note that in the papers [21, 23, 24] we have proved that only the special type
of the spectral singularities called as the essential spectral singularities
(ESS) determines the form of the spectral expansion. Therefore in Section 4 we
consider the spectral expansion by investigating the ESS of $L(V)$. Finally,
in the conclusion (Section 5) we give some conjectures, discuss the critical
points and demonstrate the shape of the spectrum and form of the spectral
expansion of $L(V)$ for all $V\in(1/2,\infty)$ assuming the validity of the
conjectures. \ 

\section{Some Definitions and Preliminary Facts}

In this section we introduce some notations, give the precise definitions of
the spectral singularity and ESS and formulate some results of the papers [7,
11, 12, 17, 19-24] as summaries which are necessary for this paper. Moreover,
here we give some results for the operator $L(V)$ which immediately and
readily follows from the results of those papers.

First let us introduce some notations in order to give the definitions of the
spectral singularities and ESS. Let $\Psi_{n,t}$ and $\Psi_{n,t}^{\ast}$ \ be
respectively the normalized eigenfunctions corresponding to the simple
eigenvalues $\lambda_{n}(t)$\ and $\overline{\lambda_{n}(t)}$ of the operators
$L_{t}(q)$ and $(L_{t}(q))^{\ast}.$ It is well-known that (see page 39 of
[10]) if $\lambda_{n}(t)$ is a simple eigenvalue of $L_{t}(q),$ then the
spectral projection $e(\lambda_{n}(t))$ defined by contour integration of the
resolvent of $L_{t}(q)$ over the closed contour containing only the eigenvalue
$\lambda_{n}(t),$ has the form%
\begin{equation}
e(\lambda_{n}(t))f=\frac{1}{d_{n}(t)}(f,\Psi_{n,t}^{\ast})\Psi_{n,t},\text{ }%
\end{equation}
where $d_{n}(t)=\left(  \Psi_{n,t},\Psi_{n,t}^{\ast}\right)  $ and $\left(
\cdot,\cdot\right)  $ is the inner product in $L_{2}[0,\pi].$ Moreover
\begin{equation}
\left\Vert e\left(  \lambda_{n}(t)\right)  \right\Vert =\left\vert
d_{n}(t)\right\vert ^{-1}.
\end{equation}

As was noted in the papers [13, 14, 4], the spectral singularity of the
operator $L(q)$ are the points $\lambda\in\sigma(L(q))$ for which the
projections of $L(q)$ corresponding to the spectral arcs $\gamma\subset
U(\lambda,\varepsilon)$ or the projections $e\left(  \lambda_{n}(t)\right)  $
of $L_{t}(q)$ corresponding to the simple eigenvalues $\lambda_{n}(t)\in
U(\lambda,\varepsilon)$ are not uniformly bounded for all $\varepsilon>0,$
where $U(\lambda,\varepsilon)=:\left\{  z\in\mathbb{C}:\left\vert
\lambda-z\right\vert <\varepsilon\right\}  $ is the $\varepsilon$-neighborhood
of $\lambda$. Therefore, by (10), we have the following definitions for the
spectral singularities in term of $d_{n}.$

\begin{definition}
A point $\lambda\in\sigma(L(q))$ is said to be a spectral singularity of
$L(q)$ if there exist $n\in\mathbb{N}$ and sequence $\{t_{k}\}\subset(-1,1]$
such that $\lambda_{n}(t_{k})\rightarrow\lambda$ and $\left\vert d_{n}%
(t_{k})\right\vert \rightarrow0$ as $k\rightarrow\infty,$ where $\lambda
_{n}(t_{k})$ is a simple eigenvalue of the operator $L_{t_{k}}(q).$ Similarly,
we say that the operator $L(q)$ has a spectral singularity at infinity if
there exist sequences $\left\{  n_{k}\right\}  \subset\mathbb{N}$ and
$\left\{  t_{k}\right\}  \subset(-1,1]$ such that $\left\vert \lambda_{n_{k}%
}(t_{k})\right\vert \rightarrow\infty$ and $\left\vert d_{n_{k}}%
(t_{k})\right\vert \rightarrow0$ as $k\rightarrow\infty,$ where $\lambda
_{n_{k}}(t_{k})$ is a simple eigenvalue.
\end{definition}

In [19] using these definitions, (10) and the well-known result of McGarvey
[8] that $L(q)$ is a spectral operator if and only if the projections of the
operators $L_{t}(q)$ are bounded uniformly with respect to $t$ in $(-\pi,\pi]$
we obtained the following.

\begin{summary}
The operator $L(q)$ is a spectral operator if and only if it has no spectral
singularities in $\sigma(L(q))$ and at infinity. Moreover $L(q)$ is an
asymptotically spectral operator if and only if it has no spectral singularity
at infinity.
\end{summary}

Thus the spectral singularities play the crucial rule for the investigations
of the spectrality of $L$. Gesztezy and Tkachenko [4] proved two versions of a
criterion for the Hill operator $L(q)$ to be a spectral operator, one analytic
and one geometric. The analytic version was stated in term of the solutions of
Hill's equation. The geometric version of the criterion uses algebraic and
geometric \ properties of the spectra of periodic/antiperiodic and Dirichlet
boundary value problems.

The problem of describing explicitly, for which potentials $q$ the Hill
operators $L(q)$ are spectral operators appeared to have been open for about
60 years. In paper [16] we found the explicit conditions on the potential $q$
such that $L(q)$ is an asymptotically spectral operator. In [24] we find a
criterion for asymptotic spectrality of the operator $L(q)$ with the potential
(7) stated in term of $a$ and $b$. Moreover, in [24], we obtained the
following result

\begin{summary}
If $ab\in\mathbb{R}$, then the operator $L(q)$ with the potential (7) is a
spectral operator if and only if it is self adjoint.
\end{summary}

It readily implies the following result.

\begin{proposition}
The operator $L(q)$ with potential (2) is a spectral operator if and only if
$V=0$, that is, $q(x)=2\cos2x$ and $L(q)$ is a self-adjoint operator.
\end{proposition}

Thus Proposition 1 shows that the theory of spectral operators can not be used
for the investigation of the non-self-adjoint operator $L(q)$ with \ the
optical potential (2).

Note that Definition 1 and Summary 2 show respectively that the boundlessness
of $\frac{1}{d_{k}}$ is the characterization of the spectral singularities and
the spectrality of $L(q).$ However, as was discovered in the papers [21, 23]
to construct the spectral expansion we need to consider the integrability of
$\frac{1}{d_{k}}$ due to the following. In [3] it was proven that in the
self-adjoint case the spectral expansion of $L(q)$ has the following elegant
form
\begin{equation}
f=\frac{1}{2\pi}\sum_{n\in\mathbb{N}}\int\nolimits_{(-1,1]}a_{n}(t)\Psi
_{n,t}dt,
\end{equation}
where
\begin{equation}
a_{n}(t)=\frac{1}{\overline{d_{n}(t)}}\left(  \int\nolimits_{\mathbb{R}%
}f(x)\overline{\Psi_{n,t}^{\ast}(x)}dx\right)  .
\end{equation}
In the non-self-adjoint case to obtain the spectral expansion, we need to
consider the integrability of $a_{n}(t)\Psi_{n,t}$ with respect to $t$ over
$(-1,1]$ which is connected with the integrability of $\frac{1}{d_{n}}$ (see
(12)). For this in [21, 23] we introduced a new notions essential spectral
singularity (ESS) which is connected with the nonintegrability of $\frac
{1}{d_{k}}$, since it may have an integrable and nonintegrable bounlessness.
\ Moreover we proved that it determines the form of the spectral expansion for
$L(q).$ That is why in this paper investigating the ESS of $L(V)$ in term of
$V$ we classify the form of its spectral expansion also in term of $V.$ In
[21] and [23] we defined the ESS as follows.

\begin{definition}
A number $\lambda_{0}\in\sigma(L)$ is said to be an essential spectral
singularity (ESS) of $L(q)$ if there exist $t_{0}\in(-1,1]$ and $n\in
\mathbb{N}$ such that $\lambda_{0}=\lambda_{n}(t_{0})$ and $\frac{1}{d_{n}}$
is not integrable over $(t_{0}-\delta,t_{0}+\delta)$ for all $\delta>0.$
Similarly, we say that the operator $L(q)$ has ESS at infinity if there exist
sequence of integers $n_{s}$ and sequence of closed subsets $I(s)$ of $(-1,1]$
such that $\lambda_{n_{s}}(t)$ for $t\in I(s)$ are the simple eigenvalues and
\begin{equation}
\lim_{s\rightarrow\infty}\int\nolimits_{I(s)}\left\vert d_{n_{s}%
}(t)\right\vert ^{-1}dt=\infty.
\end{equation}

\end{definition}

To consider the spectral expansion of $L(V)$ we use the following results of
[23] formulated here as summary.

\begin{summary}
$(a)$ The spectral expansion has the elegant form (11) if and only if $L(q)$
has no ESS and ESS at infinity.

$(b)$ If $L(q)$ has no ESS at infinity, then it has at most finite number of
ESS and its spectral expansion has the following asymptotically elegant form
\begin{equation}
f(x)=\frac{1}{2\pi}\left(  \int\nolimits_{(-\pi,\pi]}\sum\limits_{n\in
N\left(  \mathbb{E}\right)  }a_{n}(t)\Psi_{n,t}(x)dt+\sum_{n\in\mathbb{N}%
\backslash N\left(  \mathbb{E}\right)  }\int\nolimits_{(-\pi,\pi]}a_{n}%
(t)\Psi_{n,t}(x)dt\right)  ,
\end{equation}
where $N\left(  \mathbb{E}\right)  $ consists at most of a finite number of
integers and is the set of the indices $n$ for which $\Gamma_{n}$ contains at
least one ESS. Moreover%
\[
\int\nolimits_{(-\pi,\pi]}\sum\limits_{n\in N\left(  \mathbb{E}\right)  }%
a_{n}(t)\Psi_{n,t}(x)dt=\lim_{\varepsilon\rightarrow0}\sum\limits_{n\in
N\left(  \mathbb{E}\right)  }\int\nolimits_{A(\varepsilon)}a_{n}(t)\Psi
_{n,t}(x)dt,
\]
where $A(\varepsilon)=(-\pi+\varepsilon,-\varepsilon)\cup(\varepsilon
,\pi-\varepsilon).$
\end{summary}

Thus to consider the spectral expansion it is necessary to investigate the
ESS. In this paper we use the following results of [21, 24] about ESS
formulated here as summary.

\begin{summary}
$(a)$ Let $\mathbb{E},$ $\mathbb{S},$ and $\mathbb{M}$ be respectively the
sets of ESS, spectral singularities and multiple eigenvalues of $L_{t}(q)$ for
$t\in(-1,1].$ Then $\mathbb{E}\subset\mathbb{S}\subset\mathbb{M}.$

$(b)$ If $t\in(0,1)$ and $\lambda$ is a multiple eigenvalue of $L_{t}(q),$
then $\lambda$ is a spectral singularity of $L$ and is not an ESS.

$(c)$ If $\lambda$ is a multiple $2$-periodic eigenvalue with geometric
multiplicity $1,$ then it is an ESS, where the periodic and antiperiodic
eigenvalues are called as the $2$-periodic eigenvalues.

$(d)$ If $ab\neq0,$ then the operator $L(q)$ with potential (7) has no ESS at infinity.
\end{summary}

The results $(a)-(c)$ and $(d)$ were proved in [21] and [24] respectively. By
Summary 5$(d)$ we have

\begin{corollary}
If $V\neq\pm1/2,$ then $L(V)$ has no ESS at infinity.
\end{corollary}

To consider in detail the case $V\geq\sqrt{5}/2,$ that is, the case $c\geq2$
we use the following results of [17] about periodic and antiperiodic
eigenvalues formulated here as summary.

\begin{summary}
If $c\neq0,$ then the geometric multiplicity of the eigenvalues of the
operators $H_{0}(c)$, $H_{1}(c),$ $D(c)$ and $N(c),$ called as periodic,
antiperiodic, Dirichlet and Neumann eigenvalues respectively, is $1$ and the
following equalities hold
\begin{equation}
\sigma(D(c))\cap\sigma(N(c))=\varnothing,\text{ }\sigma(H_{0}(c))\cup
\sigma(H_{1}(c))=\sigma(D(c))\cup\sigma(N(c)),
\end{equation}
where $D(c)$ ($N(c))$ denotes the operator generated in $L_{2}[0,\pi]$ by (1)
with potential $2ic\cos2x$ and Dirichlet (Neumann) boundary conditions.
\end{summary}

Moreover, we use the following obvious result and the results of the papers
[11, 20] about the spectra of the operators $L(q)$ and $L_{t}(q)$ for
$t\in\lbrack0,1].$

\begin{summary}
$(a)$ If $t_{1}\in\lbrack0,1],$ $t_{2}\in\lbrack0,1]$ and $t_{1}\neq t_{2}$
then $\sigma\left(  L_{t_{1}}(q)\right)  \cap\sigma\left(  L_{t_{2}%
}(q)\right)  =\varnothing.$ It readily follows from (4).

$(b)$ \ The spectrum of $L(q)$ does not contain a closed curve (see [11]).

$(c)$ A number $\lambda_{n}(t)$ is a multiple eigenvalue of $L_{t}(q)$ of
multiplicity $p$ if and only if $p$ bands of the spectrum have common point
$\lambda_{n}(t).$ If $n\neq m,$ then the bands $\Gamma_{n}$ and $\Gamma
_{m\text{ }}$ may have only one common point (see [20]).
\end{summary}

Finally, we use the following results of [7, 12, 20] about PT symmetric potentials.

\begin{summary}
If $q(-x)=\overline{q(x)},$ then the following statements hold.

$(a)$ If two real numbers $c_{1}<c_{2}$ belong to $\Gamma_{n},$ then
$[c_{1},c_{2}]\subset\Gamma_{n}$ (see [20]).

$(b)$ If $\lambda\in\sigma(L_{t}(q))$, then $\overline{\lambda}\in\sigma
(L_{t}(q))$ (see [7])

$(c)$ If $\lambda\in\mathbb{R},$ then $F(\lambda)\in\mathbb{R},$ where $F$ is
defined in (4) (see [12]).
\end{summary}

\section{On the Bloch Eigenvalues and the Spectrum of $L(V)$}

\ First we consider the eigenvalues $\lambda_{n}(t)$ of $H_{t}(c)$ for all
$t\in\left[  0,1\right]  $, where $c\in(0,\infty).$

\begin{theorem}
All eigenvalues of the operator $H_{t}(c)$ lie on the union of the disks
\begin{equation}
D_{2c}\left(  (2n+t)^{2}\right)  :=\{\lambda\in\mathbb{C}:\left\vert
\lambda-(2n+t)^{2}\right\vert \leq2c\}
\end{equation}
for $n\in\mathbb{Z},$ where $t\in\lbrack0,1].$
\end{theorem}

\begin{proof}
Suppose to the contrary that there exists an eigenvalue $\lambda$ of
$H_{t}(c)$ lying out of $D_{2c}((2n+t)^{2})$ for all $n\in\mathbb{Z}$. Then
the inequality $\left\vert \lambda-(2n+t)^{2}\right\vert >2c$ holds for all
$n\in\mathbb{Z}$. Using it, Parseval's equality for the orthonormal basis
$\left\{  \pi^{-1/2}e^{i\left(  2n+t\right)  x}:n\in\mathbb{Z}\right\}  $ in
$L_{2}\left[  0,\pi\right]  ,$ and the following well-known relation
\begin{equation}
(\lambda-(2n+t)^{2})(\Psi,e^{i\left(  2n+t\right)  x})=(\left(  2ic\cos
2x\right)  \Psi,e^{i\left(  2n+t\right)  x}),
\end{equation}
where $\left(  \cdot,\cdot\right)  $ is the inner product in $L_{2}\left[
0,\pi\right]  $ and $\Psi$ is a normalized eigenfunction of $H_{t}(c)$
corresponding to the eigenvalues $\lambda,$ we get the following
contradiction
\[
\pi=\sum_{n\in\mathbb{Z}}\left\vert (\Psi,e^{i\left(  2n+t\right)
x})\right\vert ^{2}<\sum_{n\in Z}\frac{\left\vert (\left(  2ic\cos2x\right)
\Psi,e^{i\left(  2n+t\right)  x})\right\vert ^{2}}{\left(  2c\right)  ^{2}%
}\leq\pi.
\]

The theorem is proved.
\end{proof}

\begin{remark}
It is well-known that the operator $H_{t}(c)$ is defined by the expression
\begin{equation}
-y^{^{\prime\prime}}+2ity^{^{\prime}}+t^{2}y+\left(  2ic\cos2x\right)  y
\end{equation}
and the periodic boundary condition.\ Let us redenote the eigenvalues of
$H_{t}(c)$ and $L_{t}(V)$ by $\lambda_{n}(t,c)$ and $\lambda_{n}(t,V)$
respectively, in order to stress their dependence on $c$ and $V.$ Since
$H_{t}(c)$ analytically depends on $t$ and $c$, if $t\in\left[  0,1\right]  $
is fixed and $\lambda(t,c)$ ($\lambda(t,V)$) is a simple eigenvalue of
$H_{t}(c)$ ($L_{t}(V)$)$,$ then there exists $\varepsilon>0$ and an analytic
function $\lambda(t,\cdot)$ ($\lambda(t,\cdot)$) in $\varepsilon$-neighborhood
$U(c,\varepsilon)$ ($U(V,\varepsilon)$) of $c$ ($V$) such that $\lambda(t,c)$
($\lambda(t,V)$) is an eigenvalue of $H_{t}(c)$ ($L_{t}(V)$) for all $c\in
U(c,\varepsilon)$ ($V\in U(V,\varepsilon)$). Moreover, for fixed $c$ ($V$) the
eigenvalue $\lambda_{n}(t,c)$ ($\lambda_{n}(t,V)$) analytically depend on $t$
if it is a simple eigenvalue. Recall that, as is noted in the introduction,
for the fixed $c$ and $V$ the eigenvalues $\lambda_{n}(t,c)$ and $\lambda
_{n}(t,V)$ of $H_{t}(c)$ and $L_{t}(V)$ are numerated so that they
continuously depend on $t\in\lbrack0,1].$
\end{remark}

Now using it and Summaries 7 and 8 we prove the following.

\begin{theorem}
$(a)$ Let $c$ be a positive fixed number and $\lambda_{n}(0,c)$ be the $n$-th
periodic eigenvalue of $H_{0}(c).$ If $\lambda_{n}(0,c)$ is a real number and
$\lambda_{n}(t,c)$ are the simple eigenvalues for all $t\in\lbrack0,b),$ where
$b\in\left[  0,1\right]  $, then $\lambda_{n}(t,c)$ are the real numbers for
all $t\in\left[  0,b\right]  .$ Moreover, if $d_{n}(c)\in\lbrack0,1)$ is the
greatest positive number such that $\lambda_{n}(t,c)$ are the real eigenvalues
for all $t\in\left[  0,d_{n}(c)\right]  ,$ then $\lambda_{n}(t,c)$ are the
nonreal eigenvalues for all $t\in(d_{n}(c),1].$

$(b)$ Let $t$ be a fixed number from $\left[  0,1\right]  $ and $\lambda(t,c)$
be the simple eigenvalues of $H_{t}(c)$ for all $c\in\lbrack c_{0}(t),r(t))$
such that $\lambda(t,\cdot)$ is continuous on $\left[  c_{0}(t),r(t)\right]
,$ where $r(t)>c_{0}(t)>0$. If $\lambda(t,c_{0}(t))$ is a real (nonreal)
number, then $\lambda(t,c)$ are the real (nonreal) numbers for all
$c\in\left[  c_{0}(t),r(t)\right]  $ ($c\in\lbrack c_{0}(t),r(t))$).

$(c)$ Let $t\in\left[  0,1\right]  $ be a fixed number. If the eigenvalue
$\lambda(t,c)$ of $H_{t}(c)$ continuously depends on $c$ and changes from real
(nonreal) to nonreal (real) when $c$ moves from the left to the right of a
point $c(t)$ then $\lambda(t,c(t))$ is a multiple eigenvalue.
\end{theorem}

\begin{proof}
$(a)$ Since $c$ is fixed and $\lambda_{n}(0,c)$ is a simple eigenvalue,
$\lambda_{n}(t,c)$ analytically depends on $t$ in some neighborhood of $0$ and
there exist positive constants $\varepsilon$ and $\delta(0)$ such that the
operator $H_{t}(c)$ has a unique eigenvalue $\lambda_{n}(t,c)$ in
$U(\lambda_{n}(0,c),\varepsilon)$ whenever $\left\vert t\right\vert
<\delta(0).$ If $\lambda_{n}(t,c)$ is a nonreal number then by\ Summary
8$(b),$ $\overline{\lambda_{n}(t,c)}$ is also eigenvalue of \ $H_{t}(c)$ lying
in in $U(\lambda_{n}(0,c),\varepsilon)$ which contradicts the uniqueness of
$\lambda_{n}(t,c).$ Thus $\lambda_{n}(t,c)$ is a real number for all
$t\in\lbrack0,\delta(0)).$ Let $d$ be greatest number such that $\lambda
_{n}(t,c)$ is real for $t\in\lbrack0,d)$ and $d<b.$ Then by the condition of
the theorem and by continuity of $\lambda_{n}(\cdot,c)$ with respect to $t$
the eigenvalue $\lambda_{n}(d,c)$ is a simple and real eigenvalue. Therefore
instead of $\lambda_{n}(0,c)$ using $\lambda_{n}(d,c)$ and repeating the above
argument we obtain that $\lambda_{n}(t,c)$ is a real number for all
$t\in\lbrack0,d+\delta(d))$ which contradicts the definition of $d.$ It
implies that $\lambda_{n}(t,c)$ is a real number for all $t\in\lbrack0,b)$ and
hence, by the continuity of $\lambda_{n}(\cdot,c),$ for all $t\in\lbrack0,b].$

Now to prove the second statement of $(a)$\ suppose to the contrary that
$\lambda_{n}(t_{0},c)$ is a real eigenvalue for some $t_{0}=\left(
d_{n}(c)+\delta\right)  \in\lbrack0,1],$ where $\delta>0.$ By the definition
of $d_{n}(c)$ \ for each $0<\varepsilon<\delta$ there exists $t_{1}\in
(d_{n}(c),d_{n}(c)+\varepsilon)$ such that $\lambda_{n}(t_{1},c)$ is a nonreal
number. Then using Summary 8$(b)$ and taking into account that $\lambda
_{n}(t,c)$ continuously depend on $t,$ one can readily see that the spectrum
of $H(c)$ contains a closed curve
\[
\left\{  \lambda_{n}(t,c):t\in\left[  d_{n}(c),t_{0}\right]  \right\}
\cup\left\{  \overline{\lambda_{n}(t,c)}:t\in\left[  d_{n}(c),t_{0}\right]
\right\}  .
\]
It contradicts to the Summary 7$(b)$.

$(b)$ Instead of the fixed $c$ and variable $t$ using respectively the fixed
$t$ and variable $c$ and repeating the proof of the first statement
of$\ \left(  a\right)  $ we obtain the proof of $\left(  b\right)  $ when
$\lambda(t,c_{0})$ is a real number. If $\lambda(t,c_{0})$ is a nonreal number
and simple eigenvalue, then $\lambda(t,c)$ is a nonreal number for
$c\in\lbrack c_{0},c_{0}+\delta)$ for some $\delta>0.$ Let $d$ be greatest
number such that $\lambda(t,c)$ is nonreal for $c\in\lbrack c_{0},d)$ and
$d<r(t).$ Then $\lambda(t,d)$ is a real and simple eigenvalue. It implies that
the operator $H_{t}(c)$ has a unique eigenvalue $\lambda$ in $\varepsilon
$-neighborhood of $\lambda(t,d),$ whenever $c\in(d-\gamma(\varepsilon),d)$ for
some $\gamma(\varepsilon)>0.$ Then by the definition of $d$, $\lambda$ is a
nonreal number and by Summary 8$(b),$ $\overline{\lambda}$ is also an
eigenvalues of $H_{t}(c)$ lying in $\varepsilon$-neighborhood of
$\lambda(t,d),$ which contradicts the above uniqueness.

$(c)$ Suppose to the contrary that $\lambda(t,c(t))$ is a simple eigenvalue.
Then by $(b)$ the eigenvalue $\lambda(t,c)$ does not change from real
(nonreal) to nonreal (real) when $c$ moves from the left to the right of the
point $c(t)$. The theorem is proved.
\end{proof}

Now we consider the periodic and antiperiodic eigenvalues for $c\geq2.$ For
this first let us formulate the corresponding results of [22] for $c<2$ as the
next summary.

\begin{summary}
If $c<2$, then the operator $H_{1}(c)$ (\ $H_{0}(c))$ have $2$ eigenvalues in
the disks $D_{4}((2n-1)^{2})$ for $n=1,2,...$ (in $D_{4}((2n)^{2})$ for
$n=2,3,...)$, defined by (16) These eigenvalues are simple. One of these
antiperiodic (periodic) eigenvalues is AD (PD)and the other is AN
(PN)eigenvalue, where AD, PD, AN and PN eigenvalues are respectively the
eigenvalues belonging to $\sigma(H_{1}(c))\cap\sigma(D(c)),$ $\sigma
(H_{0}(c))\cap\sigma(D(c)),$ $\sigma(H_{1}(c))\cap\sigma(N(c))$ and
$\sigma(H_{0}(c))\cap\sigma(N(c))$ (see (15)). Moreover, in the outsides of
the disks $D_{4}((2n)^{2})$ for $n=2,3,...$ there exist two PN eigenvalues
$\lambda_{1}(0)$ and $\lambda_{2}(0)$ and one PD eigenvalue $\lambda_{3}(0)$.
\end{summary}

\begin{theorem}
$(a)$ Let $n_{1}=\left\lfloor (c+1)/2\right\rfloor ,$ where $\left\lfloor
x\right\rfloor $ denotes the integer part of the positive number $x.$ For each
$n>n_{1}$ the number of the eigenvalues (counting multiplicity) of $H_{0}(c)$
lying in $D_{2c}((2n)^{2})$ is $2$. They are the simple eigenvalues and the
real numbers. One of them is the PD and the other is the PN eigenvalue. The
number of eigenvalues (counting multiplicity) of $H_{0}(c)$ lying outsides the
disks $D_{2c}((2n)^{2})$ for $n>n_{1}$ is $2n_{1}+1.$ They lie in the
rectangle
\begin{equation}
\left\{  \lambda\in\mathbb{C}:\left\vert \operatorname{Im}\lambda\right\vert
\leq2c,\text{ }-2c\leq\operatorname{Re}\lambda\leq\left(  2n_{1}\right)
^{2}+2c\right\}  .
\end{equation}

$(b)$ Let $n_{2}=\left\lfloor c/2\right\rfloor .$ For each $n>n_{2}$ the
number of the eigenvalues (counting multiplicity) of $H_{1}(c)$ lying in
$D_{2c}((2n+1)^{2})$ is $2$. They are the simple eigenvalues and the nonreal
conjugate numbers. One of them is the AD and the other is the AN eigenvalue.
The number of the eigenvalues of $H_{1}(c)$ lying outsides of the disks
$D_{2c}((2n+1)^{2})$ for $n>$ $n_{2}$ is $2n_{2}+2.$ They lie in the
rectangle
\begin{equation}
\left\{  \lambda\in\mathbb{C}:\left\vert \operatorname{Im}\lambda\right\vert
\leq2c,\text{ }-2c\leq\operatorname{Re}\lambda\leq\left(  2n_{2}+1\right)
^{2}+2c\right\}  .
\end{equation}

\end{theorem}

\begin{proof}
$(a)$ If $n>n_{1},$ then $n>(c+1)/2$ and the disk $D_{2c}((2n)^{2})$ has no
common points with the other disks $D_{2c}((2m)^{2})$, because $\left\vert
(2m)^{2}-(2n)^{2}\right\vert >4c$ for $m\neq n.$ Hence, it follows from
Theorem 1 that the boundary of $D_{2c+\varepsilon}((2n)^{2})$ lies in the
resolvent sets of the operators $H_{0}(\alpha)$ for all $\alpha\in\left[
0,c\right]  $ if $\varepsilon$ is a sufficienly small positive number.
Therefore the projection of $H_{0}(\alpha)$ defined by contour integration
over the boundary of $D_{2c+\varepsilon}((2n)^{2})$ depends continuously on
$\alpha.$ It implies that the number of eigenvalues (counting the
multiplicity) of $H_{0}(\alpha)$ lying in $D_{2c+\varepsilon}((2n)^{2})$ are
the same for all $\alpha\in\left[  0,c\right]  $. Since $H_{0}(0)$ has two
eigenvalues in $D_{2c+\varepsilon}((2n)^{2}),$ the operator $H_{0}(c)$ has
also $2$ eigenvalues. Letting $\varepsilon$ tend to zero we obtain that
$H_{0}(c)$ has two eigenvalues (counting the multiplicity) in $D_{2c}%
((2n)^{2}).$

Instead of the operator $H_{0}(c)$ using the operators $N(c)$ and $D(c),$
taking into account that $N(0)$ and $D(0)$ have one eigenvalue in
$D_{2c}((2n)^{2})$ and repeating the above arguments we obtain that the
operators $N(c)$ and $D(c),$ have one eigenvalue in $D_{2c}((2n)^{2}).$ It
with (15) implies that the periodic eigenvalues lying in $D_{2c}((2n)^{2})$
are different numbers and hence they are simple eigenvalues. Now using Summary
1 and Theorem 2$(b)$ we obtain that they are the real numbers.

The disks $D_{2c}\left(  (2n)^{2}\right)  $ for $n=0,1,...,n_{1}$ are
contained in the rectangle (19). Moreover the rectangle (19) has no common
points with the disks $D_{2c}\left(  (2n)^{2}\right)  $ for $n>(c+1)/2.$ On
the other hand, the rectangle (19) contains $2n_{1}+1$ eigenvalues of
$H_{0}(0).$ Therefore instead of $D_{2c}((2n)^{2})$ using (19), and arguing as
above we complete the proof of $(a).$

$(b)$ Instead of $D_{2c}\left(  (2n)^{2}\right)  $ and the rectangle (19)
using $D_{2c}\left(  (2n+1)^{2}\right)  $ and the rectangle (20) respectively,
and arguing as above we get the proof of $(b).$
\end{proof}

The following theorem plays the crucial roles in the study of the spectrum of
$\ H(c).$

\begin{theorem}
Let $n_{3}=\left\lfloor (2c+1)/2\right\rfloor +1$ and $t\in(0,1).$Then for any
$b\geq(2n_{3})^{2}-2c$ the following statements hold.

$(a)$ The number of the eigenvalues of $H_{t}(c)$ lying in $(b,b+4)$ is less
than $3$.

$(b)$ The multiplicity of any eigenvalue of $H_{t}(c)$ lying in \ $\left(
b,\infty\right)  $ is less than $3.$
\end{theorem}

\begin{proof}
First, let us consider the case $t\in(0,1/2]$. If $n\in\mathbb{N}$ and $n\geq
n_{3}>\left(  2c+1\right)  /2,$ then the rectangle
\[
R_{n}(t):=\left\{  \lambda\in\mathbb{C}:\left\vert \operatorname{Im}%
\lambda\right\vert \leq2c,\text{ }(2n-t)^{2}-2c\leq\operatorname{Re}%
\lambda\leq(2n+t)^{2}+2c\right\}  .
\]
has no common points with the other rectangles $R_{m}(t)$ for $m\neq n$.
Therefore arguing as in the proof of Theorem 3$(a)$ we obtain that $R_{n}(t)$
contains two eigenvalue (counting the multiplicity). Moreover, using Theorem 1
and taking into account that the rectangle $R_{n}(t)$ contains the disks
$D_{2c}((2n-t)^{2})$ and $D_{2c}((2n+t)^{2})$ we obtain that the eigenvalues
lying in the half-plane $\left\{  \lambda\in\mathbb{C}:\operatorname{Re}%
\lambda>(2n_{3})^{2}-2c\right\}  $ belong to the rectangle $R_{n}(t)$ for some
$n\geq n_{3}.$ On the other hand, the distance between $R_{n}(t)$ and
$R_{n+1}(t)$ for $n\geq n_{3}$ is greater than $4.$ Therefore for any
$b\geq(2n_{3})^{2}-2c$ the interval $(b,b+4)$ may have common points only with
one of the rectangles $R_{n}(t)$ for $n\geq n_{3}.$ Thus this interval may
contain at most two eigenvalues of the operator $H_{t}(c).$

Instead of the rectangle $R_{n}(t)$ using the rectangle%
\[
\left\{  \lambda\in\mathbb{C}:\left\vert \operatorname{Im}\lambda\right\vert
\leq2c,\text{ }(2n+t)^{2}-2c\leq\operatorname{Re}\lambda\leq(2n+2-t)^{2}%
+2c\right\}
\]
and repeating the proof of the case $t\in(0,1/2]$ we get the proof of $(a)$ in
the case $t\in(1/2,1)$.

The proof of $(b)$ immediately follows from $(a).$
\end{proof}

Now using the above theorems, the following notations and some well-known
properties of the Hill discriminant $F(\lambda)$ defined in (4) we consider
the spectrum of $\ H(c).$

\begin{notation}
We denote the periodic eigenvalues of $H(c)$ lying in rectangle (19) and in
the disks $D_{2c}((2n)^{2})$ for $n>n_{1}$ by $\lambda_{1}(0),$ $\lambda
_{2}(0),...,\lambda_{2n_{1}+1}(0)$ and $\lambda_{2n}(0),$ $\lambda_{2n+1}(0)$
respectively, where $n_{1}$ is defined in Theorem 3$(a).$ By Theorem 3$(a),$
$\lambda_{2n}(0)$ and $\lambda_{2n+1}(0)$ for $n>n_{1}$ are the distinct real
numbers and simple eigenvalues. We denote they in increasing order
$\lambda_{2n}(0)<\lambda_{2n+1}(0)$ and hence
\begin{equation}
\lambda_{2n_{1}+2}(0)<\lambda_{2n_{1}+3}(0)<\lambda_{2n_{1}+4}(0)<\lambda
_{2n_{1}+5}(0)<....
\end{equation}
Thus from now on the periodic eigenvalues are numerated in the described
manner, the eigenvalues of $H_{t}(c)$ for $t\in(0,1],$ as is noted in Remark
1, are numerated so that $\lambda_{n}$ is a continuous function on $[0,1]$ and
the $n$-th band $\Gamma_{n}$ is defined by (6).
\end{notation}

By (4) the eigenvalues of $H_{0}(c)$ and $H_{1}(c)$ are respectively the roots
of $F(\lambda)=2$ and $F(\lambda)=-2$ and $\sigma(H(c))=\left\{  \lambda
\in\mathbb{C}:-2\leq F(\lambda)\leq2\right\}  $. Thus by Summary 8$(c)$ we
have%
\begin{equation}
\lambda\in\operatorname{Re}(\sigma(H(c)))\Longleftrightarrow\lambda
\in\mathbb{R},\text{ }(\lambda,F(\lambda))\in S(2),
\end{equation}
where $S(2)=\left\{  \lambda\in\mathbb{C}:-2\leq\operatorname{Im}\lambda
\leq2\right\}  .$ Moreover, it is well known [1] that
\begin{equation}
\lim_{\lambda\rightarrow-\infty}F(\lambda)=\infty.
\end{equation}
We also use the following well-known [1] asymptotic formulas for the
eigenvalues $\lambda_{2n}(0)$ and $\lambda_{2n+1}(0)$ lying in the disks
$D_{2c}((2n)^{2})$
\begin{equation}
\lambda_{2n}(0)=(2n)^{2}+O(n^{-1}),\text{ }\lambda_{2n+1}(0)=(2n)^{2}%
+O(n^{-1}).
\end{equation}

\begin{theorem}
$(a)$ If $\mu$ is a smallest real periodic eigenvalue, then the part of the
spectrum of $H(c)$ lying in the half-plane $\left\{  \lambda\in\mathbb{C}%
:\text{ }\operatorname{Re}\lambda<\mu\right\}  $ is nonreal.

$(b)$ Let $n=n_{1}+2,$ where $n_{1}$ is defined in Theorem 3$(a).$
Then\textit{ the part of }$\sigma(H(c))\cap\mathbb{R}$ lying in the half-space
$P(n)=:\left\{  \lambda:\operatorname{Re}\lambda>\lambda_{2n}(0)\right\}  $
\ consist of the intervals%
\begin{equation}
I_{n+1}=\left[  \lambda_{2n+1}(0),\text{ }\lambda_{2n+2}(0)\right]  ,\text{
}I_{n+2}=\left[  \lambda_{2n+3}(0),\text{ }\lambda_{2n+4}(0)\right]  ,...
\end{equation}
The \textit{gaps in the real part of the spectrum} \textit{ }lying in $P(n)$
are the intervals
\begin{equation}
\left(  \lambda_{2n}(0),\lambda_{2n+1}(0)\right)  ,\text{ }\left(
\lambda_{2n+2}(0),\lambda_{2n+3}(0)\right)  ,....
\end{equation}

\end{theorem}

\begin{proof}
$(a)$ It readily follows from (23) that the leftist point $\left(
\lambda,F(\lambda)\right)  $ of the intersection of the graph $G(F)=:\left\{
(\lambda,F(\lambda)):\lambda\in\mathbb{R}\right\}  $ with the strip $S(2)$
lies in the line $y=2,$ that is, $\lambda$ is a smallest real periodic
eigenvalue $\mu$. Therefore by (22) $\mu$ is the smallest point of
$\operatorname{Re}\left(  \sigma(H(c)\right)  .$ It means that the half-plane
$\left\{  \lambda\in\mathbb{C}:\text{ }\operatorname{Re}\lambda<\mu\right\}  $
may contain only the nonreal part of $\sigma(H(c)).$

$(b)$ Since $\lambda_{2n}(0),$ $\lambda_{2n+1}(0),...$ are real and simple
(see Notation 1) the intersection points of $G(F)$ and the line $y=2$ lying
the half-plane $\overline{P(n)}:=\left\{  \lambda:\operatorname{Re}\lambda
\geq\lambda_{2n}(0)\right\}  $ are
\begin{equation}
(\lambda_{2n}(0),2),\text{ }(\lambda_{2n+1}(0),2),\text{ }(\lambda
_{2n+2}(0),2),\text{ }(\lambda_{2n+3}(0),2),\text{ }(\lambda_{2n+4}(0),2),....
\end{equation}
On the other hand, it follows from Theorem 3$(b)$\textbf{ }that all
antiperiodic eigenvalues lying in the half-plane $\overline{P(n)}$ are nonreal
numbers due to the following. Since $\lambda_{2n}(0)$ is the eigenvalue lying
in the disk $D_{2c}((2n)^{2}),$ it readily follows from the definition of $n$
and $n_{1}$ that $2n>c+3$ and $\lambda_{2n}(0)>c^{2}+4c+9.$ If the
antiperiodic eigenvalue $\lambda(1)$ is a real, then by Theorem 3$(b)$ it
belong to the rectangle (20) and hence $\lambda(1)\leq\left(  2n_{2}+1\right)
^{2}+2c\leq\left(  c+1\right)  ^{2}+2c=c^{2}+4c+1<\lambda_{2n}(0).$ That is
why the antiperiodic eigenvalues lying in $\overline{P(n)}$ are nonreal. It
means that $G(F)$ does not intersect the line $y=-2$ in the later half-plane.
Thus we have
\begin{equation}
F(\lambda)>-2,\text{ }\forall\lambda\in\left\{  \lambda\in\mathbb{R}:\text{
}\lambda\geq\lambda_{2n}(0)\right\}  .
\end{equation}
These arguments imply that in $\overline{P(n)}$ the graph $G(H)$ may get in
and out of the strip $S(2)$ at the points (27) called respectively the points
of entry and exit. It is clear that if $\left(  \lambda_{m}(0),2\right)  $ is
a point of entry (point of exit) then $\left(  \lambda_{m+1}(0),2\right)  $ is
the point of exit (point of entry). Thus the points of entry and exit are
alternating points. It implies that if $\left[  \lambda_{m}(0),\text{ }%
\lambda_{m+1}(0)\right]  $ is the interval of $\sigma(H(c))\cap\mathbb{R}$,
then $\left(  \lambda_{m+1}(0),\lambda_{m+2}(0)\right)  $ is the gap in
$\sigma(H(c))\cap\mathbb{R}$. On the other hand, by Theorem 3$(a)$ of [20], in
the neighborhood of infinity the lengths of the intervals of $\sigma
(H(c))\cap\mathbb{R}$ approach infinity. Thus, taking into account that the
lengths of the intervals in (25) and (26) respectively approach infinity and
zero (see (24)), we get the proof of $(b)$
\end{proof}

Now using the graph of $F$ we obtain the following result.

\begin{proposition}
Suppose that $\lambda_{m}(t)$ is a simple eigenvalue for all $t\in
\lbrack0,t_{m})$ and $\lambda_{m}(0)$ is a real number. If $\lambda_{m}(0)$ is
a point of entry (point of exit) then $\lambda_{m}(t_{m})>\lambda_{m}(0)$
($\lambda_{m}(t_{m})<\lambda_{m}(0)$) and $\left[  \lambda_{m}(0),\lambda
_{m}(t_{m})\right]  \subset\Gamma_{m}$ ($\left[  \lambda_{m}(t_{m}%
),\lambda_{m}(0)\right]  \subset\Gamma_{m}$).
\end{proposition}

\begin{proof}
By Theorem 2$(a)$ $\lambda_{m}(t)$ is a real eigenvalue for all $t\in
\lbrack0,t_{m}].$ Since $\lambda_{m}(0)$ is a point of entry we have
$\lambda_{m}(t_{m})>\lambda_{m}(0)$ and by Summary 8$(a),$ $\left[
\lambda_{m}(0),\lambda_{m}(t_{m})\right]  \subset\Gamma_{m}$. In the same way
we prove the statement for the point of exit. The proposition is proved.
\end{proof}

The following remarks which readily follows from (23) will be used in the next theorems.

\begin{remark}
Let $(A,2)$ and $(B,2)$ be respectively the neighboring point of entry and
point of exit of the graph $G(F)$ defined in the proof of Theorem 5$(a)$ such
that $A$ and $B$ are the simple periodic eigenvalues and $A<B.$ Then there
exists $\varepsilon>0$ such that $F^{^{\prime}}(\lambda)<0$ and $F^{^{\prime}%
}(\lambda)>0$ respectively for $\lambda\in\left[  A-\varepsilon,A+\varepsilon
\right]  $ and $\lambda\in\left[  B-\varepsilon,B+\varepsilon\right]  .$
Assume that $F^{^{\prime\prime}}(\lambda)\neq0$ for all $\lambda\in\left[
A-\varepsilon,B+\varepsilon\right]  .$ Let $\mu_{1}$ and $\mu_{2}$ be
respectively the largest and smallest points of the interval $\left[
A+\varepsilon,B-\varepsilon\right]  $ such that $F^{^{\prime}}(\lambda)<0$ and
$F^{^{\prime}}(\lambda)>0$ for $\lambda\in\left[  A,\mu_{1}\right]  $ and
$\lambda\in\left[  \mu_{2},B\right]  .$ It is clear that both $\mu_{1}$ and
$\mu_{2}$ are the local minimum points. There are two cases:

\textbf{Case 1. }$\mu_{1}<\mu_{2}.$ Then $F$ has at least two local minimum
points $\mu_{1}$ and $\mu_{2}.$

\textbf{Case 2. }$\mu_{1}=\mu_{2}.$ Then $F$ has only one local minimum point
$\mu_{1}$ and it decreases and increases on $\left(  A,\mu_{1}\right)  $ and
$\left(  \mu_{1},B\right)  $ respectively.
\end{remark}

\begin{remark}
Since the differential equation
\begin{equation}
-y^{^{\prime\prime}}(x)+\left(  2ic\cos2x\right)  y(x)=\lambda y(x)
\end{equation}
analytically depend on the real parameters $c$ and $\lambda,$ its solutions
$\theta$ and $\varphi$ and hence the Hill discriminant $F(\lambda
,c):=\varphi^{\prime}(\pi,\lambda,c)+\theta(\pi,\lambda,c)$ analytically
depend on $\lambda$ and $c$. Moreover $F(\lambda,c)$ and its derivatives
$F^{^{\prime}}(\lambda,c)$ and $F^{^{\prime\prime}}(\lambda,c)$ with respect
to $\lambda$ continuously depend on the pair $(\lambda,c).$ If $F^{^{\prime}%
}(\lambda,c_{0})=0$ and $F^{^{\prime\prime}}(\lambda,c_{0})\neq0$ then by
implicit function theorem there exists $\varepsilon>0$ and differentiable
function $\lambda(\cdot)$ such that $F^{^{\prime}}(\lambda(c),c)=0$ for all
$c\in U(c_{0},\varepsilon).$
\end{remark}

Now we are ready to prove the main result about the spectrum of$\ H(c)$
($L(V$) for $c\in\lbrack2,\infty)$ ($V\geq\sqrt{5}/2).$ Note that the case
$c\in(0,2)$ was considered in detail in Theorems 6-8 of [22]. Therefore we
assume that $c\geq2.$

\begin{theorem}
For each $k>n_{3}$ the following statements about the \textit{bands}%
$\ \Gamma_{2k-1}$\textit{ } and $\Gamma_{2k}$ of the spectrum of$\ H(c)$ hold,
where $n_{3}$ defined in Theorem 4.

$\left(  a\right)  $ There exists unique $t_{k}\in(0,1)$ such that
$\lambda_{2k-1}(t_{k})$ is a multiple real eigenvalue lying in $I_{k}=\left[
\lambda_{2k-1}(0),\text{ }\lambda_{2k}(0)\right]  $. Its multiplicity is $2$
and $\lambda_{2k-1}(t_{k})=\lambda_{2k}(t_{k}).$

$\left(  b\right)  $ T\textit{he real parts of the bands}$\ \Gamma_{2k-1}%
$\textit{ } and $\Gamma_{2k}$ \textit{are respectively the}

\textit{ intervals }$\left\{  \lambda_{2k-1}(t):t\in\lbrack0,t_{k}]\right\}
=$ $\left[  \lambda_{2k-1}(0),\lambda_{2k-1}(t_{k})\right]  $ \textit{and }

\textit{\ }$\left\{  \lambda_{2k}(t):t\in\lbrack0,t_{k}]\right\}  =$ $\left[
\lambda_{2k}(t_{k}),\lambda_{2k}(0)\right]  .$

The\textit{ nonreal parts of }$\Gamma_{2k-1}$\textit{ } and $\Gamma_{2k}$ are
respectively\textit{ the curves}

\textit{ }$\left\{  \lambda_{2k-1}(t):t\in(t_{k},1]\right\}  $ and $\left\{
\lambda_{2k}(t):t\in(t_{k},1]\right\}  .$
\end{theorem}

\begin{proof}
$\left(  a\right)  $ First we prove that there exists multiple real eigenvalue
lying in the interval $I_{k}.$ Since the endpoints of the interval $I_{k}$ are
periodic eigenvalues we have $F\left(  \lambda_{2k-1}(0)\right)  =F\left(
\lambda_{2k}(0)\right)  =2.$ Therefore by the Roll's Theorem there exists
$\lambda\in\left(  \lambda_{2k-1}(0,c),\lambda_{2k}(0,c)\right)  $ such that
$F^{^{\prime}}(\lambda)=0.$ Moreover by Theorem 5 and Theorem 3 $\ \left(
\lambda_{2k-1}(0,c),\lambda_{2k}(0,c)\right)  $ is the interval of the real
part of the spectrum and does not contain the periodic and antiperiodic
eigenvalues. Note that for $c\geq2$ we have $n_{3}\geq n_{1}+2$ and
$n_{3}>n_{2}.$ That is why we can use Theorem 5 and Theorem 3. Therefore we
have $F(\lambda)=2\cos\pi t_{k}$ for some $t_{k}\in(0,1).$ Hence $\lambda$ is
a multiple eigenvalue of $H_{t_{k}}(c)$ lying in the interior of $I_{k}.$
Moreover by Theorem 4$(b)$ $\lambda$ is a double eigenvalue and hence
$F^{^{\prime\prime}}(\lambda,c_{0})\neq0$. Thus we have proved that there
exist double eigenvalue lying in $I_{k}.$

Now we prove the uniqueness of the double eigenvalue lying in $I_{k}.$ Suppose
that it does not hold for some $c\geq2,$ that is, there exist at least two
points $\mu_{1}$ and $\mu_{2}$ in $I_{k}$ such that $F(\mu_{1},c)$ and
$F(\mu_{2},c)$ lie in the interval $(-2,2)$ and $F^{^{\prime}}(\mu_{1},c)=0$
and $F^{^{\prime}}(\mu_{2},c)=0.$ By Theorem 4$(b)$ we have $F^{^{\prime
\prime}}(\mu_{1},c)\neq0$ and $F^{^{\prime\prime}}(\mu_{2},c)\neq0.$ Therefore
by implicit function theorem the uniqueness does not hold in some open
neighborhood of $c.$ It implies that the uniqueness holds in the closed set.
Let $c_{0}$ be the largest number such that the uniqueness holds for $0<c\leq
c_{0}.$ If $\left\lfloor c_{0}+1/2\right\rfloor +2=k$ then the uniqueness is
proved. Otherwise, there exists $\alpha>0$ such that for $c\in\left(
c_{0},c_{0}+\alpha\right)  $ the operator $H(c)$ has at least two double Bloch
eigenvalues $\mu_{1}(c)$ and $\mu_{2}(c).$ Since $F^{^{\prime}}(\lambda
,c_{0})$ is a continuous function and it has unique zero $\mu_{0}$ in the
interval $\left[  A,B\right]  ,$ where $A=\lambda_{2k-1}(0)-\beta,$
$B=\lambda_{2k}(0)+\beta$ and $\beta$ is chosen so that $F^{^{\prime}}%
(\lambda,c_{0})\neq0$ for $\lambda\in\left[  A,\lambda_{2k-1}(0)\right]  \cup$
$\left[  \lambda_{2k}(0),B\right]  ,$ there exists $\varepsilon>0$ such that
$\left\vert F^{^{\prime}}(\lambda,c_{0})\right\vert >\varepsilon$ whenever
$\lambda$ belongs to the compact $\left[  A,B\right]  \backslash(\mu_{0}%
-1,\mu_{0}+1).$ Then using the uniform continuity of $F^{^{\prime}}$ with
respect to the pair $(\lambda,c)$ we obtain that there exists $\gamma>0$ such
that $\left\vert F^{^{\prime}}(\lambda,c)\right\vert >\varepsilon/2$ whenever
$c$ and $\lambda$ belongs to $\left[  c_{0},c_{0}+\gamma\right]  $ and
$\left[  A,B\right]  \backslash(\mu_{0}-1,\mu_{0}+1)$ respectively. Here
$\gamma$ can be chosen so that $\left[  \lambda_{2k-1}(0,c),\lambda
_{2k}(0,c)\right]  \subset\left[  A,B\right]  $ for all $c\in\left[
c_{0},c_{0}+\gamma\right]  $, since the periodic eigenvalues continuously
depend on $c.$ Therefore $H(c)$ has at least two multiple Bloch eigenvalue
lying in $(\mu_{0}-1,\mu_{0}+1).$ It means that \textbf{Case 1 \ }of Remark 2
holds. Thus $F$ has two local minimum points $\mu_{1}$ and $\mu_{2}$ in
$(\mu_{0}-1,\mu_{0}+1)$. Without loss of generality it can be assumed that
$F(\mu_{1})\geq F(\mu_{2}).$ Since, $\mu_{1}$ is a local minimum point there
exists $\mu_{3}\in\left(  \mu_{1},\mu_{2}\right)  $ such that $F(\mu
_{3})>F(\mu_{1}).$ The last two inequalities with the intermediate value
theorem for the continuous function $F$ imply that there exists $\mu_{4}%
\in\lbrack\mu_{2},\mu_{3})$ such that $F(\mu_{4})=F(\mu_{1})$ (see the graph
of $F$ between the first points of entry and exit). Then the points $\left(
\mu_{1},F(\mu_{1})\right)  $ and $\left(  \mu_{4},F(\mu_{4})\right)  $ of the
graph of $F$ belong to the horizontal line $y=F(\mu_{1})=2\cos t_{1}$ for some
$t_{1}\in(0,1)$. It implies that $\mu_{1}$ and $\mu_{4}$ are the eigenvalues
of the operator $H_{t_{1}}(c).$ Moreover, by definition $\mu_{1}$ is a double
eigenvalue. Thus $H_{t_{1}}(c)$ has at least three eigenvalue (counting the
multiplicity) on $[\mu_{1},\mu_{2}].$ It contradicts Theorem 4, since
$[\mu_{1},\mu_{2}]$ is a subinterval of the interval $(\mu_{0}-1,\mu_{0}+1)$
of the length $2.$ Thus we have proved that there exist unique multiple
eigenvalue lying in $I_{k}$ and it is a double eigenvalue.

$\left(  b\right)  $ By $\left(  a\right)  $ for $t\in\left[  0,1\right]
\backslash\left\{  t_{k}\right\}  $ the eigenvalues of the operator $H_{t}(c)$
lying in

$\left[  \lambda_{2k-1}(0,c),\lambda_{2k}(0,c)\right]  $ are simple and
$H_{t_{k}}(c)$ has a double eigenvalue $\lambda(t_{k})$ lying in that
interval. Now using it Theorem 2$(a)$ and Summary 8$(a)$ we conclude that the
intervals $\left[  \lambda_{2k-1}(0),\lambda_{2k-1}(t_{k})\right]  $
\textit{and }$\left[  \lambda_{2k}(t_{k}),\lambda_{2k}(0)\right]  $ belong
respectively to the bands$\ \Gamma_{2k-1}$\textit{ } and $\Gamma_{2k}.$

Now we prove that the sets $\left\{  \lambda_{2k-1}(t):t\in(t_{k}%
,\pi]\right\}  $ and $\left\{  \lambda_{2k}(t):t\in(t_{k},\pi]\right\}  $ has
no intersections points with the real line. If $\lambda_{2k-1}(t)\in
\mathbb{R}$ for some $t\in(t_{k},\pi],$ then by Summary 8$(a)$ the interval
$\left[  \lambda_{2k-1}(t_{k}),\lambda_{2k-1}(t)\right]  $ has common
subinterval with either $\left[  \lambda_{2k-1}(0),\lambda_{2k-1}%
(t_{k})\right]  $ \textit{or }$\left[  \lambda_{2k}(t_{k}),\lambda
_{2k}(0)\right]  $. It contradict either Summary 7$(a)$ or Summary 7$(c).$
\end{proof}

Now we prove some theorems which will be used in the next sections.

\begin{theorem}
For each $V\neq\pm1/2$ the geometric multiplicity of the periodic and
antiperiodic eigenvalues of $L(V)$ is $1.$
\end{theorem}

\begin{proof}
Suppose two the contrary that the geometric multiplicity of the periodic
eigenvalue $\lambda_{n}(0)$ of $L(V)$ is $2.$ Then both solution
$\varphi(x,\lambda_{n}(a))$ or $\theta(x,\lambda_{n}(a))$ of the equation%
\begin{equation}
-y^{^{\prime\prime}}(x)+(1+2V)e^{i2x}+(1-2V)e^{-i2x}y(x)=\lambda y(x)
\end{equation}
is a periodic function. On the other hand, one easily verify that%
\begin{equation}
(1+2V)e^{i2x}+(1-2V)e^{-i2x}=2ic\cos(2x+\alpha),
\end{equation}
where $c=\sqrt{4V^{2}-1}$ and $\alpha=-\frac{\pi}{2}+\frac{i}{2}\ln\frac
{2V-1}{2V+1}.$ Note that formula (31) was found in [9], where instead of the
function $\ln$ the function $\tanh^{-1}$ was used. Moreover, the solutions
$\varphi(z,\lambda_{n}(0))$ and $\theta(z,\lambda_{n}(0))$ of the equation
obtained from (30) by replacing the real variable $x$ with the complex
variable $z$ analytically depend on $z,$ since they are the entire function of
the potential $q(z)=(1+2V)e^{i2z}+(1-2V)e^{-i2z}$ and $q$ is the entire
function of $z.$ Using these arguments we conclude that $\varphi
(x-\alpha,\lambda_{n}(0))$ or $\theta(x-\alpha,\lambda_{n}(0))$ are the
linearly independent periodic solutions of $-y^{^{\prime\prime}}%
(x)+(2ic\cos2x)y(x)=\lambda y(x).$ It contradicts Summary 6. Thus the theorem
is proved for $\lambda_{n}(0)$. In the same way we prove it for $\lambda
_{n}(1).$
\end{proof}

In [22] we proved that (see Summary 1) if $c$ is a small number, then all
periodic eigenvalues are real numbers and simple eigenvalues and hence by
Notation 1 can be numbered in increasing order $\lambda_{1}(0,c)<\lambda
_{2}(0,c)<....$ and they satisfy the formulas
\begin{equation}
\lambda_{1}(0,c)=O(c),\text{ }\lambda_{2n}(0,c)=\left(  2n\right)
^{2}+O(c)\text{ , }\lambda_{2n+1}(0,c)=\left(  2n\right)  ^{2}+O(c)\text{ }%
\end{equation}
as $c\rightarrow0$ for $n=1,2,...$ Moreover, we proved that $\lambda_{1}(0,c)$
and $\lambda_{2}(0,c)$ are the PN and $\lambda_{3}(0,c)$ is the PD eigenvalue
(see Summary 9). It means that PN eigenvalues lying in $O(c)$ neighborhood of
$4$ is less than the PD eigenvalues lying in $O(c)$ neighborhood of $4.$ Now
we proof that for each $k=1,2,$..., the eigenvalues $\lambda_{4k+1}(0,c)$ and
$\lambda_{4k+2}(0,c)$ are the PN and $\lambda_{4k}(0,c)$ and $\lambda
_{4k+3}(0,c)$ are the PD eigenvalues (see Theorem 8). These results well be
used very much in the Section 5. To consider the PD and PN eigenvalues
$\lambda$ and $\mu$ lying in $O(a)$ neighborhood of $\left(  2n\right)  ^{2}$
we use the formulas%
\begin{align}
(\lambda-4)b_{1}  &  =ab_{2},\\
\text{ }(\lambda-(2k)^{2})b_{k}  &  =ab_{k-1}+ab_{k+1},
\end{align}%
\begin{align}
(\mu-4-\frac{2a^{2}}{\mu})a_{1}  &  =aa_{2},\\
\text{ }(\mu-(2k)^{2})a_{k}  &  =aa_{k-1}+aa_{k+1},
\end{align}
where $a=ic$, $c>0$ and without loss of generality, it can be assumed that
$b_{n}=a_{n}=1.$

\begin{theorem}
If $c$ is a small number and $k\in\mathbb{N},$ then $\lambda_{4k}(0,c)$ and
$\lambda_{4k+3}(0,c)$ are the PD and $\lambda_{4k+1}(0,c)$ and $\lambda
_{4k+2}(0,c)$ are the PN eigenvalues.
\end{theorem}

\begin{proof}
To consider the PD eigenvalue $\lambda$ lying in $O(c)$ neighborhood of
$(2n)^{2}$ for $n\geq2$ we iterate $\left(  2n-1\right)  $-times formula
\begin{equation}
(\lambda-(2n)^{2})b_{n}=ab_{n-1}+ab_{n+1},
\end{equation}
which is (34) for $k=n$ by using the formulas
\begin{equation}
b_{k}=\frac{ab_{k-1}+ab_{k+1}}{\lambda-(2k)^{2}},\text{ }b_{1}=\frac{ab_{2}%
}{\lambda-4}%
\end{equation}
obtained respectively from (34) and (33) as follows. In (37) we use (38) for
$b_{n-1}$ and $b_{n+1}$. Then in the obtained formula we isolate the terms
with multiplicand $b_{n}$ and do not change they and use (38) for the terms
with multiplicand $b_{k}$ \ when $k\neq n.$ Continuing these process, $2n-1$
times we obtain
\begin{equation}
\lambda=(2n)^{2}+G_{n}(\lambda)+S_{n}(\lambda)+S_{n-1}(\lambda)+R_{n}(\lambda)
\end{equation}
where $G_{n}(\lambda),S_{n-1}(\lambda),S_{n}(\lambda)$ and $R_{n}(\lambda)$
are defined as follows. One can readily see that the $\left(  2n-1\right)
$-th usages (38) in (37) give the terms with multiplicands $b_{n}$ and
$b_{n+k},$ where $\left\vert k\right\vert \geq2.$ In (39) $R_{n}(\lambda)$
denotes the sum of the terms with multiplicand $b_{n+k}$ for $\left\vert
k\right\vert \geq2.$ Since these terms are obtained after $\left(
2n-1\right)  $-th usages (38) in (37) they contain the multiplicand $a^{2n}.$
Moreover, it follows from (38) that $b_{n+k}=O(a^{2})$ for $\left\vert
k\right\vert \geq2.$ Thus $R_{n}(\lambda)=O(a^{2n+2})$ and (39) has the form
\begin{equation}
\lambda=(2n)^{2}+G_{n}(\lambda)+S_{n}(\lambda)+S_{n-1}(\lambda)+O(a^{2n+2}),
\end{equation}
where $G_{n}(\lambda),$ $S_{n}(\lambda)$ and $S_{n-1}(\lambda)$ contains the
multiplicand $b_{n}$ and without loss of generality, it can be assumed that
$b_{n}=1.$ In (40) $G_{n}(\lambda)$ denotes the sum of the terms whose
denominators does not contain the multiplicand $\lambda-4.$ It is clear that
the terms with multiplicand $\lambda-4$ may be obtained as a result of
$\left(  2n-3\right)  $-th and $\left(  2n-1\right)  $-th usage of the second
formula of (38). In (40) $S_{n-1}(\lambda)$ and $S_{n}(\lambda)$ are the sums
of the terms whose denominators contain the multiplicand $\lambda-4$ and are
obtained in $\left(  2n-3\right)  $-th and $\left(  2n-1\right)  $-th usages
the second formula of (38) respectively. It is clear that $S_{n-1}$ is
obtained by using the formulas in (38) in the following order
$k=n-1,n-2,....,2,1,2,...,n-1.$ Therefore it has the form
\begin{equation}
S_{n-1}=\frac{a^{2n-2}}{(\lambda-4)}%
{\textstyle\prod\limits_{k=2}^{n-1}}
\left(  \lambda-(2k)^{2}\right)  ^{-2}.
\end{equation}

Similarly iterating the formulas (36) $2n-1$ times and arguing as above we get%
\begin{equation}
\mu=(2n)^{2}+G_{n}(\mu)+\widetilde{S}_{n-1}(\mu)+\widetilde{S}_{n}%
(\mu)+O(a^{2n+2}).
\end{equation}
Here $\widetilde{S}_{n-1}(\lambda)$ and $\widetilde{S}_{n}(\lambda)$ are
obtained respectively from $S_{n-1}(\lambda)$ and $S_{n}(\lambda)$ by
replacing the $\lambda-4$ with $\lambda-4-\frac{2a^{2}}{\lambda}$.
$\ \ $Therefore, using (32) and taking into account that $S_{n}(\lambda)$ and
$\widetilde{S}_{n}(\lambda)$ are obtained in$\ \left(  2n-1\right)  $-th
usages the second formula of (38) and (35) respectively, and hence they
contains the multiplicand \ $a^{2n}$ we get
\begin{equation}
\widetilde{S}_{n}(\lambda)=S_{n}(\lambda)+O(a^{2n+2}).
\end{equation}
In (41) replacing $\lambda-4$ with $\lambda-4-\frac{2a^{2}}{\lambda}$ we get
\begin{equation}
\widetilde{S}_{n-1}(\lambda)=S_{n-1}(\lambda)+\frac{2a^{2n}}{\left(
\lambda-4\right)  ^{2}\lambda}%
{\textstyle\prod\limits_{k=2}^{n-1}}
\left(  \lambda-(2k)^{2}\right)  ^{-2}+O(a^{2n+2}).
\end{equation}
Now in the second term of the right-hand side of (44) instead of $\lambda$
writing $\lambda=(2n)^{2}+O(a^{2})$ (it readily follows from (40)) we obtain%
\begin{equation}
\widetilde{S}_{n-1}(\lambda)=S_{n-1}(\lambda)+\frac{2a^{2n}}{\left(
(2n)^{2}-4\right)  ^{2}(2n)^{2}}%
{\textstyle\prod\limits_{k=2}^{n-1}}
\left(  (2n)^{2}-(2k)^{2}\right)  ^{-2}+O(a^{2n+2}).
\end{equation}
\ Now subtracting equality (42) from (40) and using (43) and (45) we obtain
\begin{equation}
\lambda-\mu=f(\lambda)-f(\mu)-\frac{2a^{2n}}{\left(  (2n)^{2}-4\right)
^{2}(2n)^{2}}%
{\textstyle\prod\limits_{k=2}^{n-1}}
\left(  (2n)^{2}-(2k)^{2}\right)  ^{-2}+O(a^{2n+2}),
\end{equation}
where $f$ $(\lambda)=S_{n}(\lambda)+S_{n-1}(\lambda).$ It is clear that $f$
$^{^{\prime}}(\lambda)=O(a^{2})$ for $\lambda=(2n)^{2}+O(a^{2}).$ Therefore by
the mean value theorem we have $f(\lambda)-f(\mu)=\left(  \lambda-\mu\right)
(1+O(a^{2}))$. Now using it in (46) and taking into account that $a=ic,$
$c>0$we get
\begin{equation}
\lambda-\mu=\frac{-2\left(  -1\right)  ^{n}c^{2n}}{\left(  (2n)^{2}-4\right)
^{2}(2n)^{2}}%
{\textstyle\prod\limits_{k=2}^{n-1}}
\left(  (2n)^{2}-(2k)^{2}\right)  ^{-2}+O(a^{2n+2}).
\end{equation}
It is clear that the multiplicands $(2n)^{2}-(2k)^{2}$ for $k=2,3,...,(n-1)$
are the positive number. It follows from (47) that if $n=2k,$ then the PD
eigenvalue $\lambda$ lying in the neighborhood of $\left(  2n\right)  ^{2}$ is
less than the PN eigenvalues $\mu$ lying in the neighborhood of $\left(
2n\right)  ^{2}.$ On the other hand by Notation 1 the periodic eigenvalues
lying in $\left(  2n\right)  ^{2}$ are denoted as $\lambda_{2n}(0,c)<\lambda
_{2n+1}(0,c).$ Thus $\lambda_{4k}(0,c)$ and $\lambda_{4k+1}(0,c)$ are the PD
and PN eigenvalues respectively. Instead of $n=2k$ using $n=2k+1$ and
repeating the above proof we obtain that $\lambda_{4k+2}(0,c)$ and
$\lambda_{4k+3}(0,c)$ are the PN and PD eigenvalues respectively.
\end{proof}

\begin{corollary}
If $c$ is a small number, then for each $k\in\mathbb{N}$, the $(2k-1)$-th
entry and exit points of the graph $G(F)=:\left\{  (\lambda,F(\lambda
)):\lambda\in\mathbb{R}\right\}  $ to the strip $S(2)$ are the PN eigenvalues
and the $2k$-th entry and exit points are the PD eigenvalues.
\end{corollary}

Now we consider the antiperiodic eigenvalues for the small value of $c$ which
also will be used in Section 5. By Summary 1 all antiperiodic eigenvalues for
small $c$ are nonreal numbers. Moreover, by Theorem 3$(b)$ for the small value
of $c$ the disk $D_{2c}((2n-1)^{2})$ contains one AD and one AN eigenvalues
$\lambda(c)$ and $\mu(c)$ and
\begin{equation}
\lambda(c)=\left(  2n-1\right)  ^{2}+O(c),\text{ }\mu(c)=\overline{\lambda
(c)}.
\end{equation}

\begin{theorem}
For the AD eigenvalue $\lambda(c)$ lying in $D_{2c}((2n-1)^{2})$ the formula
\begin{equation}
\operatorname{Im}\left(  \lambda(c)\right)  =C_{n}a^{2n-1}+O(c^{2n})
\end{equation}
holds as $c\rightarrow0,$ where $a=ic$ and $C_{n}$ is a positive number.
\end{theorem}

\begin{proof}
Formula (49) for $n=1$ follows from formula (24) of [22]. To prove it for
$n>1$ we iterate $\left(  2n-2\right)  $-times the formula%
\begin{equation}
(\lambda(c)-(2n-1)^{2})c_{n}=ac_{n-1}+ac_{n+1}%
\end{equation}
(see (18) of [22] for $k=n)$ as follows. Each time isolate the terms with
multiplicand $c_{n}$ (we call they as isolated terms) and use the formulas
\begin{equation}
c_{1}=\frac{ac_{2}}{\lambda(c)(a)-1-a},\text{ }c_{k}=\frac{ac_{k-1}+ac_{k+1}%
}{\lambda(c)-(2k-1)^{2}},\text{ }%
\end{equation}
where $k=2,3,...,$ for the terms with multiplicand $c_{k}$ \ when $k\neq n.$
After $\left(  2n-2\right)  $ times usages of (51) in (50) we obtain
\begin{equation}
\lambda(c)=(2n+1)^{2}+G_{n}(\lambda(c))+S_{n}(\lambda(c))+O(a^{2n}),
\end{equation}
where $S_{n}$ is obtained by using the formulas (51) for $c_{k}$ in the
following order $k=n-1,n-2,....,2,1,2,...,n-1$ and hence has the form
\begin{equation}
S_{n}=\left(  (\lambda(c)-1-a)^{-1}%
{\textstyle\prod\limits_{k=2}^{n-1}}
(\lambda(c)-(2k-1))^{-2}\right)  a^{2n-2}.
\end{equation}
The sum of other isolated terms is denoted by $G_{n}(\lambda(c)).$ Thus
$G_{n}(\lambda(c))$ is the sum of fractions whose numerators are $a^{2k}$ for
$k=1,2,...,(n-1),$ denominators are the products of $\lambda(c)-(2s-1)^{2}$
for $s\neq n$ and hence
\begin{equation}
\overline{G_{n}(a,\lambda(c))}=G_{n}(a,\overline{\lambda(c)}),\text{
}\overline{G_{n}(a,\lambda(c))}-G_{n}(a,\lambda(c))=\left(  \overline
{\lambda(c)}-\lambda(c)\right)  O(a^{2}).
\end{equation}
Using (48) we obtain $(\lambda(c)-1-a)^{-1}=\gamma(1+a\gamma)^{-1}%
=\gamma(1+a+O(a^{2})),$ where $\gamma=(\lambda(c)-1)^{-1}.$ \ It with (53)
implies that
\begin{equation}
\operatorname{Im}S_{n}=\left(  ((2n-1)^{2}-1)^{-1}%
{\textstyle\prod\limits_{k=2}^{n-1}}
((2n-1)^{2}-(2k-1))^{-2}\right)  a^{2n-1}+O(a^{2n})
\end{equation}
$\ $ Now using (54) and (55) in (52) we get the proof of (49).
\end{proof}

\section{Spectal Singularities, ESS, and Spectral Expansion}

First using the results of Section 3 we consider spectral singularities and ESS.

\begin{theorem}
If $V>1/2$ then the operator $L(V)$ has infinitely many spectral singularities
and spectral singularity at infinity.
\end{theorem}

\begin{proof}
By Theorem 6$(a)$ for each $k>n_{3}=\left\lfloor (2c+1)/2\right\rfloor +1$
there exists $t_{k}\in(0,\pi)$ such that $\lambda_{2k-1}(t_{k})=\lambda
_{2k}(t_{k})$ is a double eigenvalue. It is spectral singularity due to
Summary 5$(b).$ \ Thus all bands $\Gamma_{k}$ for $k>n_{3}$ contains a
spectral singularity. On the other hand, it readily follows from Definition 1
that if $L(V)$ has a sequence of spectral singularities converging to infinity
then it has the spectral singularity at infinity.
\end{proof}

To consider the ESS we use the following theorem.

\begin{theorem}
A number $\lambda$ is an ESS of $L(V)$ if and only if it is either a multiple
periodic or a multiple antiperiodic eigenvalue.
\end{theorem}

\begin{proof}
Let $\lambda\in S(L(V))$ be an ESS. Then by Summary 5$(a)$ it is multiple
Bloch eigenvalue. Since, by Summary 5$(b),$ the Bloch eigenvalues $\lambda
_{n}(t)$ for $t\in(0,\pi)$ and $n\in\mathbb{N}$ are not ESS, $\lambda$ is
either periodic or antiperiodic eigenvalue. Now suppose that $\lambda\in
S(L(V))$ is a multiple periodic or antiperiodic eigenvalue. Then by Theorem 7
the geometric multiplicity of $\lambda$ is $1$ and by Summary 5$(c)$ it is an ESS.
\end{proof}

\begin{remark}
Note that by (8) the operators $H(c)$ and $L(V)$ have the same Hill
discriminants $F(\lambda)$ and $G(\lambda)$ if $c=\sqrt{4V^{2}-1}.$ Therefore
if $\lambda$ is a multiple eigenvalue of $H_{t}(c),$ then it is also a
multiple eigenvalue of $L_{t}(V)$ with the same multiplicity. Note that if
$V<1/2,$ then $ic\in(-\infty,0)$ \ and the well known self-adjo\i nt
Mathieu-Hill operator $H(c)$ has no double Bloch eigenvalues. Therefore, by
Summary 5$(a)$ for \ $V<1/2$ the operator $L(V)$ has no spectral singularities
and ESS. If $V=1/2$ then all $\ 2$-periodic eigenvalues of $L(V)$ except $0$
are double eigenvalue with geometric multiplicities 1 (see [5]) and hence are
ESS (see Summary 5$(c)$).
\end{remark}

Now we consider the ESS of $L(V)$ for $V>1/2.$

\begin{theorem}
Let $1/2<V<\sqrt{5}/2.$ $(a)$ If$\ V\neq V_{2},$ then the operator $L(V)$ has
no ESS.

$(b)$ If $V=V_{2}$ , then $L(V)$ has a unique ESS and it is $\lambda
_{1}(0)=\lambda_{2}(0).$
\end{theorem}

\begin{proof}
$(a)$ By Summary 1 all periodic and antiperiodic eigenvalues are simple.
Therefore by Theorem 11 the operator $L(V)$ has no ESS.

$(b)$ By Summary 1 $\lambda_{1}(0)=\lambda_{2}(0)$ is a unique double periodic
eigenvalue and all antiperiodic eigenvalues are simple. Therefore by Theorem
11 $\lambda_{1}(0)$ is a unique ESS.
\end{proof}

\begin{theorem}
There exists a sequence $0<c_{2}<c_{3}<\cdot\cdot\cdot$ such that
$c_{k}\rightarrow\infty$ as $k\rightarrow\infty$ and $H(c)$ has no ESS and has
ESS respectively if and only if $c\neq c_{k}$ for all $k$ and $c=c_{k}$ for
some $k,$ where $c_{2}=\sqrt{4V_{2}^{2}-1}$\ and $V_{2}$ is the second
critical point. Moreover, the number of ESS of $H(c_{k})$ is not greater than
$c_{k}+2.$ The number of the bands of $H(c_{k})$ containing at least one ESS
is not greater than $2c_{k}+3$ respectively.
\end{theorem}

\begin{proof}
Let $c<2n-1.$ Then using Theorem 3$(a)$ and $(b)$ we conclude that the total
number of multiple periodic and antiperiodic eigenvalues (counting
multiplicity) is not greater than $4n-1.$ It means that there exist at most
$s$ different numbers denoted by $\rho_{1}(c),\rho_{2}(c),...\rho_{s}(c)$
which are the multiple periodic and antiperiodic eigenvalues, where
$s\leq2n-1.$ Since the operators $H_{0}(c)$ and $H_{1}(c)$ analytically depend
on $c,$ by the well-known perturbation theory if $\rho_{k}(c_{0})$ is a
multiple periodic eigenvalue then there exists $\varepsilon_{k}$ such that the
eigenvalues of the operators $H_{0}(c)$ and $H_{1}(c)$ for $c\in
U(c_{0},\varepsilon_{k})$ lying in the small neighborhood of $\rho_{k}(c_{0})$
are simple. Therefore there exist at most finite number $c_{1},c_{2}%
,...,c_{m}$ from $\left(  0,2n-1\right)  $ such that $H_{0}(c)$ and $H_{1}(c)$
may have multiple eigenvalue.

Let $n_{k}$ be the smallest integer such that $c_{k}<2n_{k}-1$. Instead of $c$
and $n$ using $c_{k}$ and $n_{k}$ respectively and repeating the above
argument we obtain that the total number of multiple $2$-periodic (periodic
and antiperiodic) eigenvalues without counting multiplicity and counting
multiplicity are not greater than $2n_{k}-1$ and $4n_{k}-1$ respectively. By
the definition of $n_{k}$ we have $2n_{k}-1\leq c_{k}+2.$ Therefore by Theorem
11 and Summary 7$(c)$ the number of the ESS of the operator $H(c_{k})$ and the
number of the bands of $H(c_{k})$ containing at least one ESS are not greater
than $c_{k}+2$ and $2c_{k}+3$ respectively.
\end{proof}

\begin{remark}
It was suitable to formulate Theorem 13 in term of the parameter $c$, since we
have used the notation of Theorem 3. By (8) and Theorem 13 the operator $L(V)$
for $V>1/2$ has no ESS and has ESS respectively if and only if $V\neq V_{k}$
for all $k$ and $V=V_{k}$ for some $k,$ where $V_{k}=\frac{1}{2}\sqrt
{c_{k}^{2}+1}$ is said to be the $k$-th critical point. Moreover the number of
ESS of $L(V_{k})$ and the number of the bands of $L(V_{k})$ containing at
least one ESS is not greater than $\sqrt{4V_{k}^{2}-1}+2$ and $2\sqrt
{4V_{k}^{2}-1}+3$ respectively. Denote by $\Gamma_{1},\Gamma_{2}%
,...,\Gamma_{m(k)}$ the bands containing at least one ESS of $L(V_{k})$, where
$m(k)\leq2\sqrt{4V_{k}^{2}-1}+3.$ Then the bands $\Gamma_{m(k)+1},$
$\Gamma_{m(k)+2},...,$ don't contain an ESS.
\end{remark}

Now using Summary 4, Corollary 1 and Theorem 12 we get the following results
about the spectral expansions of $L(V)$ for $1/2<V<\sqrt{5}/2.$

\begin{theorem}
Let $1/2<V<\sqrt{5}/2.$

$(a)$ If$\ V\neq V_{2}$, then the spectral expansion for $L(V)$ has the
elegant form (11).

$(b)$ If $V=V_{2}$ , then the spectral expansion for $L(V)$ has the following
form
\begin{equation}
f(x)=\frac{1}{2\pi}\left(  \int\limits_{(-\pi,\pi]}\left[  a_{1}(t)\Psi
_{1,t}(x)+a_{2}(t)\Psi_{2,t}(x)\right]  dt+\sum_{n=3}^{\infty}\int
\nolimits_{(-\pi,\pi]}a_{n}(t)\Psi_{n,t}(x)dt\right)  .
\end{equation}
Moreover
\[
\int\limits_{(-\pi,\pi]}a_{1}(t)\Psi_{1,t}+a_{2}(t)\Psi_{2,t}dt=\lim
_{\varepsilon\rightarrow0}\left(  \int\limits_{(-\pi,\pi]\backslash
(-\varepsilon,\varepsilon)}a_{1}(t)\Psi_{1,t}+\int\limits_{(-\pi
,\pi]\backslash(-\varepsilon,\varepsilon)}a_{2}(t)\Psi_{2,t}dt\right)
\]

\end{theorem}

\begin{proof}
$(a)$ By Theorem 12$(a)$ and Corollary 1 the operator $L(V)$ has no ESS and
ESS at infinity. Therefore the proof follows from Summary 4$(a)$

$(b)$ By Theorem 12$(b),$ $\lambda_{1}(0)=\lambda_{2}(0)$ is a unique ESS of
$L(V)$. \ Therefore the set $N\left(  \mathbb{E}\right)  $ defined in Summary
4$(b)$ is $\left\{  1,2\right\}  $. Moreover, by Corollary 1 the operator
$L(V)$ has no ESS at infinity. Therefore (56) follows from (14).
\end{proof}

Now changing the variable to $\lambda$ in (11) and (56) as was done in [19]
and using the relations
\[
\Gamma_{n}=\lim_{\varepsilon\rightarrow0}\Gamma_{n}(\varepsilon),\text{
}\Omega_{1}(\varepsilon)=\Gamma_{1}(\varepsilon)\cup\Gamma_{2}(\varepsilon),
\]
where $\Gamma_{n}(\varepsilon)=\{\lambda=\lambda_{n}(t):t\in\left[
\varepsilon,\pi-\varepsilon\right]  \}$ we obtain

\begin{theorem}
Let $1/2<V<\sqrt{5}/2.$

$(a)$ If$\ V\neq V_{2}$, then the spectral expansion for $L(V)$ has the form
\begin{equation}
f(x)=\frac{1}{\pi}\sum_{k\in\mathbb{N}}\left(  \int\limits_{\Gamma_{k}}%
(\phi(x,\lambda))\frac{1}{p(\lambda)}d\lambda\right)  ,
\end{equation}
where
\[
\phi(x,\lambda)=\theta^{^{\prime}}h(\lambda)\varphi(x,\lambda)+\frac{1}%
{2}(\theta-\varphi^{^{\prime}})(h(\lambda)\theta(x,\lambda)+g(\lambda
)\varphi(x,\lambda)-\varphi g(\lambda)\theta(x,\lambda),
\]%
\[
h(\lambda)=\int\limits_{-\infty}^{\infty}\varphi(x,\lambda)f(x)dx,\text{
}g(\lambda)=\int\limits_{-\infty}^{\infty}\theta(x,\lambda)f(x)dx,\text{
}p(\lambda)=\sqrt{4-F^{2}(\lambda)},
\]
$\varphi=\varphi(\pi,\lambda),$ $\varphi^{^{\prime}}=\varphi^{^{\prime}}%
(\pi,\lambda),$ $\theta=\theta(\pi,\lambda)$ and $\theta^{^{\prime}}%
=\theta^{^{\prime}}(\pi,\lambda).$

$(b)$ If $V=V_{2}$ , then the spectral expansion for $L(V)$ has the following
form
\begin{equation}
f(x)=\frac{1}{\pi}p.v.\left(  \int\limits_{\Omega_{1}}(\phi(x,\lambda
))\frac{1}{p(\lambda)}d\lambda\right)  +\frac{1}{\pi}\sum_{k=3}^{\infty
}\left(  \int\limits_{\Gamma_{k}}(\phi(x,\lambda))\frac{1}{p(\lambda)}%
d\lambda\right)  ,
\end{equation}
where $p.v.$ integral over $\Omega_{1}$ is the limit of the integral over
$\Omega_{1}(\varepsilon)$ as $\varepsilon\rightarrow0.$
\end{theorem}

Now instead of Theorem 12 using Theorem 13 and Remark 5 and repeating the
proof of Theorems 14 and 15 we get the following spectral expansion of $L(V)$
for all $V>1/2.$

\begin{theorem}
Suppose that $V>1/2$.

$(a)$ If $V\neq V_{k}$ for $k\geq2$, then the spectral expansion for $L(V)$
has the form (11). Moreover, the spectral expansion for $L(V)$ in term of
$\lambda$ has the form (57).

$(b)$ If $V=V_{k},$ then the operator $L(V)$ has the spectral expansion
\begin{equation}
f(x)=\frac{1}{2\pi}\left(  \int\nolimits_{(-\pi,\pi]}\left[  \sum
\limits_{n=1}^{m(k)}a_{n}(t)\Psi_{n,t}(x)\right]  dt+\sum_{n=m(k)+1}^{\infty
}\int\nolimits_{(-\pi,\pi]}a_{n}(t)\Psi_{n,t}(x)dt\right)  ,
\end{equation}
where $m(k)$ is defined in Remark 5. Moreover, the spectral expansion for
$L(V_{k})$ in term of $\lambda$ can be obtained from (58) by replacing
$\Omega_{1}$ and $\sum_{k=3}^{\infty}$ with $%
{\textstyle\bigcup\nolimits_{n=1}^{m(k)}}
\Gamma_{n}$ \ and $\sum_{n=m(k)+1}^{\infty}$ respectively.
\end{theorem}

\section{Conclusions}

\textbf{ }In this section first we discuss the similarities and differences
between the noncritical\textbf{ }case\textbf{ }($V\neq V_{k}$ for all
$k=2,3,...)$ and the critical case ($V=V_{k}$ for some $k\geq2)$, where
$V>1/2$ and the critical points $V_{2}<V_{3}<....$ are defined in Remark 5.
Then we illustrate the changes of the spectrum and spectral expansion of
$L(V)$ when $V$ changes from $1/2$ to $\infty$ and give some conjectures.

\textbf{Similarities: }By Theorem 10 in the both cases the operator $L(V)$ has
infinitely many spectral singularities and spectral singularity at infinity.
Hence by Summary 2 in the both cases the operator $L(V)$ is not an
asymptotically spectral operator and hence is not a spectral operator. Note
that by Definition 1 the spectral singularities is connected by the
boundlessness of the function $\frac{1}{d_{k}}$ and in the both
(critical\textbf{ }and noncritical) cases they are the unbounded function.

\textbf{Differences:} The detailed investigation of the spectral
singularities, namely, the study of the nonintegrability of $\frac{1}{d_{k}}$
helps us to see the differences between the critical and noncritical
cases,\ since the unbounded functions $\frac{1}{d_{k}}$ may became as
integrable as well as nonintegrable. It was the reason to introduce the new
types of spectral singularities called as ESS (see Definition 2) which is
defined by the nonintegrability of the functions $\frac{1}{d_{k}}$. Thus the
study only the boundlessness of $\frac{1}{d_{k}}$ or equivalently of the
projections $e(\lambda_{n}(t))$ does not explain the differences of these
cases and the properties of the critical case. By Theorems 12 and 13 and
Remark 5 the operator $L(V)$ in the critical case\textbf{ }$V=V_{k}$ has ESS,
while\textbf{ }in the noncritical case $V\neq V_{k}$ for $k=2,3,...$ has no
ESS. Moreover in the critical case $V=V_{k}$ the existence of the ESS does not
allow to be the spectral decomposition of the operator $L(V_{k})$ in the
elegant form (11) (see Theorems 14 and 16), while the spectral expansion of
the operator $L(V)$ for $V\neq V_{k}$ has an elegant form. Note that in the
critical\textbf{ }case\textbf{ }$V=V_{2}$ the spectral decomposition of the
operator $L(V)$ has no elegant form, because the functions $a_{1}(t)\Psi
_{1,t}$ and $a_{1}(t)\Psi_{1,t}$ have nonintegrable singularities at $t=0.$
Similarly in the critical\textbf{ }cases\textbf{ }$V=V_{k}$ for $k>2$ the
spectral decomposition of the operator $L(V_{k})$ has no elegant form, because
the functions $a_{n}(t)\Psi_{n,t}$ for $n=1,2,...,m$ has nonintegrable
singularities either at $t=0$ or at $t=\pi.$ However their sum is integrable
over $(-\pi,\pi]$. It means that, in the spectral expansions (56) and (59) the
bracket comprising the functions with indices $n=1,2$ and $n=1,2,...,m$
respectively is necessary. Note also that, if we consider the spectral
expansion in term of $\lambda$, then we need to use the $p.v.$ integral about
ESS, since the integrals about this point do not exist. We do not need the
$p.v.$ integral if and only if $V\neq V_{k}$ for $k=2,3,...$ . Moreover, one
can obtain a spectral expansion without parenthesis and $p.v.$ integrals if
and only if $V\neq V_{k}$ for $k=2,3,...$, that is, only in the
noncritical\textbf{ }case\textbf{.} It is a principle difference between the
noncritical and critical cases.\textbf{ }

Now we consider the changes of the spectrum and spectral expansion of $L(V)$
when $V$ changes from $1/2$ to $\infty.$ First we assume that the following
two principles formulated as conjectures hold.

\begin{conjecture}
(\textbf{Leaving principle.)} If $c$ increases from $0$ to $\infty,$ then all
Bloch eigenvalues of the operator $H(c)$ with the pure imaginary potential
$2ic\cos2x$ leave the real line. In other words, if $V$ increases from $1/2$
to $\infty,$ then all Bloch eigenvalues of $L(V)$ leave the real line.
\end{conjecture}

\begin{conjecture}
(\textbf{Irreversibility principle.)} The Bloch eigenvalues never came back to
the real line after leaving it.
\end{conjecture}

\textbf{The discussion of the} \textbf{Leaving principle. }Conjecture 1 was
proved in [22] for the Bloch eigenvalues $\lambda_{1}(t,V)$ and $\lambda
_{2}(t,V)$. Moreover, in [22] we have proved that if $V$ moves from $1/2$ to
$V_{2},$ where $V_{2}$ is the second critical point, then the all antiperiodic
eigenvalues $\lambda_{n}(1,V)$ leave he real line while the periodic
eigenvalues $\lambda_{n}(0,V)$ moves over real line (see Summary 1). Since
$\lambda_{n}(t,V)$ continuously depends on $t$ after antiperiodic eigenvalues
$\lambda_{n}(1,V)$ the Bloch eigenvalues $\lambda_{n}(t,V)$ for $t$ close to
$1$ leave the real line.

The periodic eigenvalues do not leave the real line for $V\in\left(
1/2,V_{2}\right)  $ due to the following. By Theorem 2$(c)$ if the eigenvalue
$\lambda(t,V)$ changes from real to nonreal when $V$ moves from the left to
the right of the constant$\ V(t)>1/2$, then $\lambda(t,V(t))$ is a multiple
eigenvalue, where $t$ is a fixed number from $\left[  0,1\right]  .$ The first
periodic eigenvalues is simple for $V\in\left(  1/2,V_{2}\right)  $ and hence
it can not leave the real line until being the multiple eigenvalue. As was
proven in [22] (see Case 1 and Summary 1 in the introduction), if $V$ moves
from the left to the right of $V_{2}$ then the first and the second
eigenvalues get close to each other for $V<V_{2},$ the equality $\lambda
_{1}(0,V_{2})=\lambda_{2}(0,V_{2})$ holds and they leave the real line for
$V>V_{2}$. Moreover the last equality is possible since both $\lambda
_{1}(0,V_{2})$ and $\lambda_{2}(0,V_{2})$ are the PN eigenvalues (see Summary
9 and (15)).

Now we show that the eigenvalue $\lambda_{3}(0,V)$ moving over real line may
become the double eigenvalues if and only if $\lambda_{3}(0,V)=\lambda
_{4}(0,V)$. If $V\in(V_{2},\sqrt{5}/2),$ then on the left side of $\lambda
_{3}(0,V)$ there is not a real periodic eigenvalue, since $\lambda_{1}(0,V)$
and $\lambda_{2}(0,V)$ are nonreal eigenvalues (see Summary 1). Moreover by
Summary 9 and Theorem 8, $\lambda_{3}(0,a)$ and $\lambda_{4}(0,a)$ are PD
eigenvalues while $\lambda_{5}(0,a)$ is a PN eigenvalue and lies on the right
of $\lambda_{4}(0,a).$ Thus by moving over real line the eigenvalue
$\lambda_{3}(0,a)$ have no possibility to coincide with the eigenvalues lying
on the left of $\lambda_{3}(0,a)$ and $\lambda_{4}(0,a)$ have no possibility
to coincide with the eigenvalues lying on the right of $\lambda_{4}(0,a),$
since PN and PD eigenvalue are different (see (15)). Using these arguments and
taking into account that all Bloch eigenvalues sooner or later leave the real
line we conclude that there exists $V_{3}$\ such that $\lambda_{3}%
(0,V_{3})=\lambda_{4}(0,V_{3}).$ Moreover by Summary 1 we have $V_{3}\geq
\frac{\sqrt{5}}{2}>V_{2}.$ In the same way we conclude that there exist a
number $V_{k}$ such that if $1/2<V<V_{k},$ $V=V_{k}$ and $V>V_{k},$ then
$\left(  2k-3\right)  $-th and $\left(  2k-2\right)  $-th periodic eigenvalues
$\lambda_{2k-3}(0,V)$ and $\lambda_{2k-2}(0,V)$ are real numbers,
$\lambda_{2k-3}(0,V_{k})=\lambda_{2k-2}(0,V_{k})\in\mathbb{R},$ and both of
them are nonreal respectively. Moreover Theorem 3$(a)$ implies that
$V_{k}\rightarrow\infty$ as $k\rightarrow\infty.$ Thus relying on these
arguments, Theorem 3$(a)$ and Theorem 13 we believe that the following
conjecture holds.

\begin{conjecture}
There exists a sequence $V_{2}<V_{3}<...$ of real numbers, called the critical
points, approaching infinity such that the following hold. If $V_{k}%
<V<V_{k+1},$ then $\lambda_{1}(0,V),$ $\lambda_{2}(0,V),...,\lambda
_{2k-2}(0,V)$ are the nonreal eigenvalues and $\lambda_{2k-1}(0,V),\lambda
_{2k}(0,V),...$ are the real simple eigenvalues. If $V=V_{k+1},$ then
$\lambda_{1}(0,V),$ $\lambda_{2}(0,V),...,\lambda_{2k-2}(0,V)$ are the nonreal
eigenvalues, $\lambda_{2k-1}(0,V)=\lambda_{2k}(0,V)$ is a real double
eigenvalue and $\lambda_{2k+1}(0,V),$ $\lambda_{2k+2}(0,V),...$ are the real
simple eigenvalues.
\end{conjecture}

These conjecture for $k=2$ was proved in [22] and is explained in the
introduction. Let us give some explanation for this conjecture. The
perturbation $2ic_{2}\cos2x$ of norm $2c_{2},$ where $c_{2}$ $=\sqrt
{4V_{2}^{2}-1}$ is enough to move the first and second eigenvalues $0$ and $4$
of the unperturbed operator $H_{0}(0)$ so that $\lambda_{1}(0,c_{2})$ and
$\lambda_{2}(0,c_{2})$ coincide, since $2c_{2}>2$. However this perturbation
is not enough to move the third and \ forth eigenvalue $4$ and $16$ of
$H_{0}(c)$ so that the equality $\lambda_{3}(0,c)=\lambda_{4}(0,c)$ holds for
some $c\in(0,2)$ since the distance between $4$ and $16$ is $12.$ Therefore
$V_{3}>V_{2}.$ The same explanation show that $V_{k+1}>V_{k}$. Moreover, by
Theorem 3$(a)$ and Notation 1 $\lambda_{2k}(0,c)$ is a real eigenvalue if
$c<k-1.$ It implies that $c_{k}\geq k-1$, where $c_{k}$ $=\sqrt{4V_{k}^{2}%
-1},$ and hence $V_{k}\rightarrow\infty$ as $k\rightarrow\infty.$

Now let us explain why the Bloch eigenvalues $\lambda_{n}(t,V)$ leave the real
line for all $t\in(0,1).$ In [22] it was proved if $1/2<V$ $<V_{2},$ then
$\operatorname{Re}(\Gamma_{2k-1}\cup\Gamma_{2k})=\left[  \lambda
_{2k-1}(0),\lambda_{2k}(0)\right]  $ (see Case 1 in introduction). It implies
that $\lambda_{2k-1}(t)$ and $\lambda_{2k}(t)$ lie between $\lambda_{2k-1}(0)$
and $\lambda_{2k}(0)$ if they are real numbers. On the other hand, by
Conjecture 3 the eigenvalues $\lambda_{2k-1}(0)$ and $\lambda_{2k}(0)$ leave
the real line when $V$ moves to the right of $V_{k+1}.$ Therefore the
eigenvalues $\lambda_{2k-1}(t)$ and $\lambda_{2k}(t)$ leave the real line for
all $t\in\lbrack0,1]$ and hence as a result no parts of $\Gamma_{2k-1}$ and
$\Gamma_{2k}$ remain real when $V$ moves to the right of $V_{k+1}.$ Moreover,
it is clear that the eigenvalues $\lambda_{n}(t,V)$ leave the real line with
decreasing order of $t$ from $1$ to $0.$ These arguments and Theorem 6 imply
the following conclusions about the spectrum of $L(V)$.

\begin{conclusion}
Suppose that conjectures 1-3 hold and $V>1/2$. Then

$(a)$ For each $k\geq2$ the spectrum of $L(V_{k})$ has the following properties:

\textbf{Separation:} The bands $\Gamma_{2s-3}$ and $\Gamma_{2s-2}$ for $s<k$
are nonreal separated curves symmetric with respect to the real line.

\textbf{Connection by the end point:} The bands $\Gamma_{2k-3}$ and
$\Gamma_{2k-2}$ are connected by $\lambda_{2k-3}(0),$ that is, $\Gamma
_{2k-3}\cap\Gamma_{2k-2}=\left\{  \lambda_{2k-3}(0)\right\}  $.

\textbf{Connection by the interior point:} The bands $\Gamma_{2s-3}$ and
$\Gamma_{2s-2}$ for $s>k$ are connected by double eigenvalue $\lambda
_{2s-3}(t)=\lambda_{2s-2}(t),$ where $t\in(0,\pi),$ that is, $\Gamma
_{2k-3}\cap\Gamma_{2k-2}=\left\{  \lambda_{2s-3}(t)\right\}  $.

$(b)$ If $V_{k}<V<V_{k+1},$ then \ the \textbf{Separation }and
\textbf{Connection by the interior point }$\lambda_{2s-3}(t)=\lambda
_{2s-2}(t)$ occurs for $s\leq k$ and $s>k$ respectively. In this case
\textbf{Connection by the end points} does not occurs.

$(c)$ The number $V_{k}$ is a point after which no parts of the $\left(
2k-3\right)  $-th and $\left(  2k-2\right)  $-th bands are real. That is why
it is natural to call it as the $k$-th critical point.
\end{conclusion}

\textbf{Explanation of the conjectures by using the graph of the Hill
discriminant }$F.$ First let us explain Summary 1 in term of the graph. In the
case $V<V_{2}$ the first point of entry and the first point of exit of the
graph $G(F)$ defined in the proof of the Theorem 5 are $\left(  \lambda
_{1}(0),2\right)  $ and $\left(  \lambda_{2}(0),2\right)  $ respectively. The
part $\left\{  (\lambda,F(\lambda)):\lambda\in\left[  \lambda_{1}%
(0),\lambda_{2}(0)\right]  \right\}  $ of the graph $G(F),$ called as the
first part, consists of the points $\left(  \lambda_{1}(t),2\cos t\right)  $
and $\left(  \lambda_{2}(t),2\cos t\right)  $ for real $\lambda_{1}(t)$ and
$\lambda_{2}(t)$. If $V$ approaches $V_{2}$ from the left then the first and
second periodic eigenvalues get close to each other and the real eigenvalues
$\lambda_{1}(t)$ and $\lambda_{2}(t)$ become nonreal numbers (\textbf{see Case
1-Case 3} in introduction). It means that the first part of the graph of $F$
is leaving the strip $S(2)$. In other word the graph of $F$ is rising up. As
$V$ reaches $V_{2}$ we get the equality $\lambda_{1}(0)=\lambda_{2}(0)$ which
means that the first point of entry $\left(  \lambda_{1}(0),2\right)  $ and
the first point of exit $\left(  \lambda_{2}(0),2\right)  $ coincides and the
line $y=2$ becomes tangent line to the curve $y=F(\lambda)$ at the point
$\left(  \lambda_{1}(0),2\right)  .$ If $V$ moves to the right of $V_{2},$
then the eigenvalues $\lambda_{1}(0)$ and $\lambda_{2}(0)$ get off the real
line. Therefore the first part of the graph of $F$ completely get out of the
strip $S(2)$.

In general Corollary 2 shows that for each $k$ both the first coordinates of
the $k$-th points of entry and exit are either PN or PD eigenvalues. Therefore
they may get close to each other coincides and get out of the strip as $V$
increases and moves from the left to the right of $V_{k+1}.$ It means that the
$k$-th part of the graph of $F$ get out of the strip $S(2)$ and by Theorem
6$(b)$ no parts of the $\left(  2k-1\right)  $-th and $2k$-th bands remain real.

Let us stress also the following. As is noted in above first the antiperiodic
eigenvalues $\lambda(1,V)$ leave he real line and after antiperiodic
eigenvalues the Bloch eigenvalues $\lambda(t,V)$ for $t$ close to $1$ leave
the real line. On the other hand, in the end of the proof of Theorem 6$(a)$ we
noted that the graph of $F$ between the $k$-th point of entry and the $k$-th
point of exit has the form as it is sketched in the picture of Section 3
between the second points of entry and exit. It implies that when the points
of entry and exit get close to each other, then the eigenvalues $\lambda(t,V)$
leave the real line with decreasing order of $t$ from $1$ to $0.$ Finally the
periodic eigenvalues live the real line. Moreover, Theorem 5$(a)$ implies that
if all periodic eigenvalues lying in the left of some number $\mu$ leave the
real number then all spectrum of $L(V)$ on the left of $\mu$ left the real line.

Now we discuss the spectral expansion by assuming that the following
conjecture holds.

\begin{conjecture}
(\textbf{Simplicity principle.)} The nonreal periodic and antiperiodic
eigenvalues are the simple eigenvalues.
\end{conjecture}

Let us give some explanation for this conjecture. The nonreal periodic
(antiperiodic) eigenvalues $\lambda_{2k-1}(0)$ and $\lambda_{2k}(0)$
($\lambda_{2k-1}(1)$ and $\lambda_{2k}(1)$) are complex conjugate numbers and
by Conjecture 2 they do not came back to real line. That is why they can not
coincide. On the other hand, Theorem 8 shows that if $\lambda_{2k-1}(0)$ is PN
(PD) eigenvalues lying in upper (lower) half-plane then the neighboring
nonreal complex periodic eigenvalues $\lambda_{2k-3}(0)$ and $\lambda
_{2k+1}(0)$ lying in the same half-plane are PD (PN) eigenvalues. Therefore
$\lambda_{2k-1}(0)$ can not be multiple eigenvalue due to the neighboring
eigenvalues. Similarly it follows from Theorem 9 and Summary 1 that if
$\lambda_{2k-1}(1)$ is AD (AN) eigenvalues lying in upper (lower) half-plane
then the neighboring nonreal antiperiodic eigenvalues $\lambda_{2k-3}(0)$ and
$\lambda_{2k+1}(0)$ lying the same half-plane are AN (AD) eigenvalues.
Therefore $\lambda_{2k-1}(1)$ also can not be multiple eigenvalue due to the
neighboring eigenvalues. On the other hand the overlapping of the unneighbored
eigenvalues $\lambda_{2k-1}(1)$ and $\lambda_{2k+3}(1)$ is the unlikely event,
since they must pass through the $\left(  2k+1\right)  $-th band
$\Gamma_{2k+1}(1)$ which is the continuous curve lying between these
eigenvalues. That is why Conjecture 4 holds with high probability.

By Summary 1 and Conjectures 2 and 4 for $V>1/2$ the operator $L(V)$ has no
multiple antiperiodic eigenvalues. Moreover by Summary 1 and Conjectures 3 and
4 if $V\neq V_{k}$ for all $k\geq2$ and $V=V_{k+1}$ then $L(V)$ has no
multiple periodic eigenvalues and has only one double periodic eigenvalue
$\lambda_{2k-1}(0)=\lambda_{2k}(0)$ respectively. Therefore by Summary 4 and
Theorems 11 and 16 the spectral expansion of $L(V)$ has the form

\begin{theorem}
Suppose that Conjectures 1-4 hold and $V>1/2$. Then for each $k\geq2$ the
spectral expansion for the operator $L(V_{k+1})$ has the following form
\begin{equation}
f=\frac{1}{2\pi}\left(  \int\nolimits_{(-\pi,\pi]}\left[  a_{2k-1}%
(t)\Psi_{2k-1,t}+a_{2k}(t)\Psi_{2k,t}\right]  dt+\sum_{n\in\mathbb{N}%
\backslash\left\{  2k-1,2k\right\}  }\int\nolimits_{(-\pi,\pi]}a_{n}%
(t)\Psi_{n,t}dt\right)
\end{equation}
and
\begin{equation}
f(x)=\frac{1}{\pi}p.v.\left(  \int\limits_{\Omega_{k}}(\phi(x,\lambda
))\frac{1}{p(\lambda)}d\lambda\right)  +\frac{1}{\pi}\sum_{n\in\mathbb{N}%
\backslash\left\{  2k-1,2k\right\}  }\left(  \int\limits_{\Gamma_{k}}%
(\phi(x,\lambda))\frac{1}{p(\lambda)}d\lambda\right)  .
\end{equation}

If $V\neq V_{k}$ for all $k\geq2,$ then the spectral expansion for $L(V)$ has
the elegant forms (11) and (57).
\end{theorem}

Thus if the conjectures hold, then the spectral decomposition of the operator
$L(V_{k+1})$ has no elegant form, because the functions $a_{2k-1}%
(t)\Psi_{2k-1,t}$ and $a_{2k}(t)\Psi_{2k,t}$ has nonintegrable singularities
at $t=0$. It means that, in the spectral expansion (60) the bracket comprising
the functions with indices $2k-1$ and $2k$ corresponding to ESS $\lambda
_{2k-1}(0)=\lambda_{2k}(0)$ is necessary. \ Besides, if we consider the
spectral expansion in term of $\lambda$, then it is necessary to use the
$p.v.$ integral about ESS $\lambda_{2k}(0)$, since the integrals about this
point do not exist. That is why\ only for the integral over the component
$\Omega_{k}=\Gamma_{2k-1}\cup\Gamma_{2k}$ containing the ESS $\lambda_{2k}(0)$
we use the $p.v.$ integral (see (61)).

\end{document}